\documentclass[a4paper,10pt]{article}
\usepackage{graphicx}
\usepackage{amsmath,amssymb,amsfonts,amsthm}
\usepackage{url, enumerate,anysize}
\usepackage{algorithm}
\usepackage{algorithmic}
\usepackage{lscape}
\usepackage[subnum]{cases}

\usepackage[usenames,dvipsnames,svgnames,table]{xcolor}
\usepackage{color}
\usepackage{natbib}
\usepackage[colorlinks=true,linkcolor=violet, citecolor=NavyBlue, urlcolor=blue]{hyperref}
\usepackage{enumitem}   

\usepackage{graphicx}
\usepackage{mathtools}
\usepackage{tikz}
\usepackage[titletoc]{appendix}
\newtheorem{thm}{Theorem}[section]

\newtheorem{lem}[thm]{Lemma}
\newtheorem{defn}[thm]{Definition}
\newtheorem{rmk}[thm]{Remark}
\newtheorem{prop}{Proposition}

\usepackage[usenames,dvipsnames,svgnames,table]{xcolor}

\title{A unifying computations of Whittle's Index for Markovian bandits}
\author{Urtzi Ayesta, Manu K. Gupta\footnote{Corresponding author: Manu K. Gupta (manugupta@linuxmail.org)}, and Ina Maria Verloop\\IRIT, 2 rue C. Camichel, Toulouse, France}

\marginsize{1.5cm}{1.5cm}{1.5cm}{1.5cm}

\begin{document}
\maketitle

\begin{abstract}
The multi-armed restless bandit framework allows to model a wide variety of decision-making problems in areas as diverse as industrial engineering, computer communication, operations research, financial engineering, communication networks etc. In a seminal work, Whittle developed a methodology to derive well-performing  (Whittle's) index policies that are obtained by solving a relaxed version of the original problem. However, the computation of Whittle's index itself is a difficult problem and hence researchers focused on calculating Whittle's index numerically or with a problem dependent approach. 

In our main contribution we derive an analytical expression for Whittle's index for any Markovian bandit with both finite and infinite transition rates. We  derive sufficient conditions for the optimal solution of the relaxed problem to be of threshold type, and obtain conditions for the bandit to be indexable, a property assuring the existence of Whittle's index.  Our solution approach provides a unifying expression for Whittle's index, which we highlight by retrieving known indices from literature  as particular cases. The applicability of finite rates is illustrated with the machine repairmen problem, and that of  infinite rates by an example of communication networks where transmission rates react instantaneously to packet losses.
\end{abstract}

\textbf{keywords:} Restless bandits, Whittle index.

\section{Introduction}

Markov Decision Processes (MDPs) provide a mathematical framework for sequential decision making where outcomes are random. Formally, an MDP is a sequential stochastic control process, where a decision maker aims at minimizing its long term cost. The basic setup is as follows, at each time step, a state dependent cost is accrued, the decision maker chooses an action among the available ones, and the process randomly moves to a new state. Due to their broad applicability, MDPs are found in many areas, including artificial intelligence, economics, and operations research. 

An MDP can be solved via dynamic programming, however, this is a computationally intractable task for realistic model sizes. As a result, classes of MDPs that are analytically tractable have received a lot of attention. One such class is the Multi-Armed Bandit Problem  (MABP) framework. In an MABP, there are multiple concurrent projects or bandits. The decision maker knows the states of all bandits and the cost in every state, and aims at minimizing the average cost. At every decision epoch, the decision maker needs to select one bandit, the state of this selected bandit evolves stochastically, while the states of all other bandits remain \textit{frozen}.  In a ground-breaking result, Gittins showed that the optimal policy that solves an MABP is an index rule, nowadays referred to as Gittins' index policy \citep{gittins2011multi}. Thus, for each bandit, one calculates Gittins' index, which depends only on its own current state and stochastic evolution. The optimal policy activates the bandit with highest current index in each decision epoch.

The Restless Multi-Armed Bandit Problem (RMABP), introduced in \cite{whittle1988restless}, is a more general class of problems, in which  the states of non-selected bandits also evolve randomly. That is, in contrary to an MABP, here the non-selected bandits do not remain \textit{frozen}. RMABPs have become extremely popular over the years, and have been applied in many contexts, including  inventory routing, machine maintenance, health-care systems, outsourcing warranty repair, etc.  RMABPs can not be solved analytically, except for some toy examples. 
Whittle developed a methodology to obtain heuristics by  solving a relaxed version of the RMABP. The obtained heuristics, nowadays known as Whittle's index policy, rely on calculating Whittle's index for each of the bandits, and activating in every decision epoch the bandit with highest Whittle's index. It has been reported on numerous instances that Whittle's index policy provides strikingly good performance, and it has been shown to be asymptotically optimal as the number of bandits grows large.

Fundamental questions regarding Whittle's index policy concern their existence and their complexity in computation. To prove existence,  one needs to establish a technical property known as \textit{indexability}.  Computing  Whittle's index  might be involved, and in practice the indices are computed   on a problem-to-problem basis (see more in Section~\ref{sec:relatedwork}), either numerically or analytically.

In this paper, we focus on the average performance criterion (as presented in the original paper by Whittle). We allow both finite and infinite rates at which a bandit changes state. We first present a general algorithm that, although computationally demanding, determines whether the problem is indexable  and calculates the index for any RMABP and allows multi-dimensional state space evolution. For a one-dimensional bandit, we then derive sufficient conditions for a problem to be indexable and for the optimal solution to a RMABP to be of threshold type. Then, in our main result, we show that if threshold policies are optimal and if a certain monotonicity condition holds,  Whittle's index can be expressed as a function of the steady-state distribution. We apply our unifying characterization to several problems considered in the literature to show that they are indexable, that threshold policies are optimal, and finally to retrieve Whittle's index as a direct application of our unifying  analytical expression.

\subsection{Related literature}
\label{sec:relatedwork}

A classical reference for MDPs is \cite{puterman2014markov}, and a comprehensive coverage for MABP and RMABP is given in \cite{gittins2011multi}.  Book length treatments of restless bandits can be found in \cite{jacko2010dynamic} and \cite{ruiz2008indexable}. A discussion on the nearly-optimal performance of Whittle's index is given in \cite{nino2007dynamic}, and asymptotic optimality of Whittle's index as the number of bandits grows is shown in \cite{weber1990index, verloop2016asymptotically}. Further restless bandit formulation has been used in diverse domains (see \cite{james2016developing}, \cite{group}).

For average cost criterion and with a countable state space, \cite{nino2006restless} gives sufficient conditions for an RMABP to be indexable and provides an analytical expression (same as ours) for the Whittle index. One of the   conditions consists in showing that the so-called marginal workloads are strictly positive for any set of states. This ensures that optimal policies are of threshold type. 
Instead, in this paper,  we provide an algorithm that does not require threshold optimality, as might be of interest when a bandit lives in a multi-dimensional state space and threshold optimality might be impossible to establish. 
In addition, in the case when threshold optimality can be established independently, we show that a weaker assumption on the marginal workloads is sufficient for the results to hold. 

Even though Whittle's seminal work introduced Whittle's index within the context of average cost criterion, a large body of work has focused on tackling an RMABP under the total discounted cost criterion.  For the discounted cost criterion and finite state space, \cite{nino2007dynamic} provides a thorough analysis and efficient algorithms (based on linear programming) to establish indexability and to give an expression for Whittle's Index. The same approach was undertaken to obtain the Whittle's index for a general RMABP in~\cite{nino2006restless}.


As explained in the introduction, the analytical computation of the index has mostly been carried out on a problem-to-problem basis. The main idea to calculate them  is to sweep the state space, by recursively identifying and calculating the states with higher Whittle's indices.  This can be done by iterative schemes as for example in {\cite{borkar2017whittle} for a processor sharing queue, \cite{borkar2017index} in a problem of cloud computing, \cite{borkar2017opportunistic} for a scheduling problem in a wireless setting, and \cite{pattathil2017distributed} in the context of content delivery networks.} And in some particular cases analytically, see for example \cite{argon2009dynamic} for a load balancing problem with dedicated arrivals, \cite{opp2005outsourcing} for  outsourcing warranty repairs, \cite{ayer2019prioritizing} for Hepatitis C Treatment in US Prisons, and \cite{TON_maialen} for restless bandits that are of birth and death type.
Another popular approach in the literature to calculate the index for the average cost criterion has been to calculate first the index for the discounted case, and then let the discounting factor tend to one. This is the approach undertaken in  e.g.\ \cite{glazebrook2005index} for the machine repairman problem, 
\cite{ansell2003whittle} for a multi-class queue with convex holding costs, and \cite{nino2002dynamic} for a queue with admission control. 
A feature that renders the discounted problem more amenable is that the dynamic programming equation of the MDP has only one unknown, the value function, whereas for the average cost criterion, the dynamic programming equation has two unknowns, one being the value function and the other one the average performance (\cite{puterman2014markov}). 

In this paper, we take a direct approach and work directly with the average cost criterion. This will allow us to obtain a unifying framework to write an analytical expression for the Whittle's index.
As we will explain in Sections~\ref{finite_transitions_examples}~and~\ref{infinite_transition_rates}, all applications mentioned above for which Whittle's index was found are particular instances of our unifying approach.   Note that this includes the examples for which  Whittle's index was so far only calculated by iterative numerical scheme. 

\section{Model description}\label{sec:model}
We consider an RMABP with $K$ ongoing projects or bandits. At any moment in time, bandit~$k$, $k=1,...,K$, is in a certain state~$n_k \in \mathcal{N}^d$, with $d\in \mathcal{N}^+$.   Decision epochs are defined as the moments when one of the bandits changes its state. At each decision epoch, the controller decides for each bandit to either make the bandit passive,  action $a=0$, or to make the bandit active, action $a=1$.  

Throughout this paper, we consider bandits that are modeled as a continuous-time Markov chain, that is, when bandit $k$ is in state $n_k$, it changes the state after an exponentially distributed amount of time. Transition rates  for bandit~$k$, which can either be finite or infinite, depend only on the bandits' state~$n_k$ and the action chosen for this bandit. 
Let ${\mathcal{I}_k(n_k,a)}$ be an indicator function for the event of an infinite (impulse) transition  for bandit~$k$  at state $n_k$ under action $a$, i.e., 
 \[
{\mathcal{I}_k(n_k,a)} := 
  \begin{cases}
1 \text{ if  state changes instantaneously},\\
    0 \text{ otherwise}. \\     
  \end{cases}
\]
If ${\mathcal{I}_k(n_k,a)} =1$, let $p_k^a(n_k,m_k)$ be the probability of making an immediate transition  to state~$m_k$ from state $n_k$. 
If ${\mathcal{I}_k(n_k,a)} =0$,  let $q_k^a(n_k, m_k)$ be the   finite transition rate of going from state $n_k$ to $m_k$  under action $a\in \{0,1\}$. Note that  the state  of  a bandit  can evolve both when being active and passive, and this does not depend on the states/actions of other bandits. Hence, given the action taken, the dynamics of each bandit is independent of the others.


A policy $\phi$ decides which bandit is made active. Because of the Markovian property, we  focus on   policies which base their decision only on the current states of the bandits. For policy~$\phi$, $N_k^\phi(t)$ denotes the state of bandit $k$ at time $t$, and $\vec N^\phi(t)=(N_1^\phi(t),\ldots, N_K^\phi(t))$. 
Let $S_k^\phi(\vec{N}^\phi(t)) \in \{0,1\}$ represent whether or not  bandit $k$ is made active at time $t$, that is, it equals 1 if the bandit is activated, and 0 otherwise.

For bandit~$k$, let $f_k(n_k,a)$ be a function of state $n_k$ and action $a$ and assume it is bounded by a polynomial in $n_k$. Further assume that if  $\mathcal{I}_k(n_k,a)=1$, then $f_k(n_k,a)=0$. 
A policy~$\phi$ is called feasible if it satisfies the following constraint: 
\begin{equation}\label{main_const}
\sum_{k=1}^Kf_k(N_k^\phi(t), S_k^\phi(\vec{N}(t)))   \le M,
\end{equation}
where $M$ is some given constant.
That is,~\eqref{main_const} gives  a constraint on the number of activated/passive bandits resulting  in finite transition rates. We denote by $\mathcal{U}$ the set of policies that satisfy the constraint~(\ref{main_const}) and make the system ergodic. We assume that such a policy exists, hence $\mathcal{U}$ is non-empty.  

Since $f_k(n_k,a)=0$ if $\mathcal{I}_k(n_k,a)=1$, it is direct that   there is no  constraint on the number of impulses at a given time~$t$ (we refer to Remark~\ref{rem} for a   discussion on the control of impulses).
For non-impulsive control,  setting $f_k(N_k^\phi, S_k^\phi(\vec{N})) = S_k^\phi(\vec{N})$, the constraint (\ref{main_const}) reduces to $\sum\limits_{k=1}^K S_k^\phi(\vec{N}) \le M$, that is, at most $M$ out of $K$ bandits can be made active. This is a slight variation  of the   classical restless bandit problem, where exactly $M$ bandits need to be  activated at any time~$t$\footnote{This can be shown by introducing so-called dummy bandits with zero cost and fixed state, see~\cite{verloop2016asymptotically}.}.
Another interesting function is when $f_k(N_k^\phi, S_k^\phi(\vec{N}))$ represents the expected capacity occupation (volume), in which case the constraint (\ref{main_const}) reduces to the family of sample-path knapsack capacity allocation constraints as recently explored in \cite{jacko2016resource}, \cite{graczova2014generalized}.

Let $C_k(n_k,a)$ denote the cost per unit of time when bandit $k$ is in state $n_k$ and is either passive ($a=0$) or active ($a=1$). Let $L_k^\infty(n_k, m_k, a)$ be the lump cost  when bandit $k$ under action~$a$ immediately changes state from $n_k$ to $m_k$. We assume that both $C_k(n_k,a)$ and $L_k^\infty(n_k, m_k, a)$ can be bounded by a polynomial in~$n_k$ and $m_k$.
Instead of a lump cost $L_k^\infty(n_k, m_k, a)$, we will work with {$C_k^{\infty, \phi}(\vec{n}, a)$, which is defined as the cost per unit of time due to impulse transitions when policy $\phi$ is implemented, we are in state ${n}$ and action $a$ is taken. This is given by
\begin{eqnarray*}
C_k^{\infty, \phi}(\vec{n},a) := &&\sum_{\tilde{n}_k}\sum_{m_k}\left( q_k^{a}(n_k, \tilde{n}_k) \times \mathcal{I}_{k}(\tilde{n}_k, S_k^\phi(\vec{M}_{k}(\vec{n}, \tilde{n}_k)))\right.
\\
&&\left.\times{ p_k^{S_k^\phi(\vec{M}_{k}(\vec{n}, \tilde{n}_k))}(\tilde{n}_k,m_k) \times  L_k^\infty(\tilde{n}_k, m_k, S_k^\phi(\vec{M}_{k}(\vec{n}, \tilde{n}_k)) }\right),
\end{eqnarray*}
where $\vec{M}_{k}(\vec{n}, \tilde{n}_k)$ equals the state $\vec n$ in which the $k$th component of $\vec{n}$ is replaced by $\tilde{n}_k$.
}


 This can be seen as follows. The cost per unit of time due to impulse transitions 
consists of the transition rate at which you go to an impulse situation, times its corresponding lump cost. Regarding the first, we have such a transition for some bandit~$k$, which has a transition to some state~$\tilde{n}_k$, hence the new state is $\vec{M}_{k}(\vec{n}, \tilde{n}_k)$, in which it experiences an impulse. Hence, it concerns a transition rate {$q_k^{a}(n_k, \tilde{n}_k)$} if and only if $\mathcal{I}_{k}(\tilde{n}_k, S_k^\phi(\vec{M}_{k}(\vec{n}, \tilde{n}_k)))=1$. Now, regarding the corresponding lump cost, this cost depends on the state~$m_k$ where bandit~$k$ ends up after the impulse. That is, we need to multiply probability $p_k^{S_k^\phi(\vec{M}_{k}(\vec{n}, \tilde{n}_k))}(\tilde{n}_k,m_k)$ with the corresponding lump cost $L_k^\infty(\tilde{n}_k, m_k, S_k^\phi(\vec{M}_{k}(\vec{n}, \tilde{n}_k))$.

The objective is to find a scheduling policy $\phi \in \mathcal{U}$ that minimizes the long-run average cost: 
\begin{equation}\label{objective}
C^\phi :=\limsup_{T\rightarrow\infty}\sum_{k=1}^K\frac{1}{T}\mathbb{E}\left( \int_{0}^T C_k(N_k^\phi(t), S_k^\phi(\vec{N}^\phi(t)))+ C_k^{\infty,\phi}(\vec{N}^\phi(t), S_k^\phi(\vec{N}^\phi(t))) dt \right).
\end{equation}
 The first term is a contribution from holding cost per unit of time and the second term corresponds to the lump cost due to impulses.

\begin{rmk}
\label{rem}
There is no    sample-path constraint on the impulse control, that is, on the number of impulses allowed at each moment in time, as this will be by default satisfied for any policy (see (\ref{main_const})).
Instead, in order to control the number of impulses in the system, one can set the lump cost for infinite transitions, $L_k^\infty(n,m,a)$, appropriately. Another method would be to include a constraint on the \emph{time-average} number of impulses. 
The later is beyond the scope of the paper, as this would imply deriving an index-based heuristic that satisfies a time-average constraint,  and, to the best of our knowledge,  no index heuristics have been proposed for such  settings.  
\end{rmk}

\section{Lagrangian Relaxation and Whittle's index policy}\label{relaxation}

Finding a policy that minimizes the long run average cost~(\ref{objective}) under constraint~(\ref{main_const}) is  intractable in general. In fact, it is shown in \cite{papadimitriou1999complexity} that the restless bandit problems are 
PSPACE complete, which is much stronger evidence of intractability than NP-hardness.  
Following~\cite{whittle1988restless}, a very fruitful approach has been to study the relaxed problem in which the constraint~\eqref{main_const} is replaced by its time-averaged version, that is,
\begin{equation}\label{average_const}
\limsup_{T\rightarrow\infty}\frac{1}{T}\mathbb{E}\left( \int_{0}^T\sum_{k=1}^K f_k(N_k^\phi(t), S_k^\phi(\vec{N}^\phi(t)))dt \right)\le M.
\end{equation}
Let $\mathcal{U}^{REL}$ be the set of stationary policies $\phi$ that satisfy (\ref{average_const}) and for which the Markov chain is ergodic. Note that the set of policies that make the relaxed problem ergodic includes $\mathcal{U}$, i.e., the
set of ergodic policies for the original problem. The objective of the relaxed problem is hence to determine a policy that solves (\ref{objective}) under constraint (\ref{average_const}). An optimal policy for the relaxed
problem then serves as a heuristic for the original optimization
problem.
 
 Using the Lagrangian approach, we write the relaxed problem as the following unconstrained problem: find a policy $\phi$ that minimizes $\mathcal{C}^\phi(W) :=${\small
\begin{equation}\label{relaxation} 
\limsup_{T\rightarrow\infty}\frac{1}{T}\mathbb{E}\left( \int_{0}^T \left(\sum_{k=1}^KC_k(N_k^\phi(t), S_k^\phi(\vec{N}^\phi(t)))\right.\right. + C_k^{\infty, \phi}(\vec{N}^\phi(t), S_k^\phi(\vec{N}^\phi(t)))
\left.
-  W\left(\sum_{k=1}^Kf_k(N_k^\phi(t), S_k^\phi(\vec{N}^\phi(t))) -M )\right)dt \right),
\end{equation}}
where $W$ is the Lagrange multiplier. The latter can be decomposed into $K$ subproblems, one for each bandit $k$, that is, minimize  {\small
\begin{equation}\label{subproblem_k}
\mathcal{C}_k^\phi:= \limsup_{T\rightarrow\infty}\frac{1}{T}\mathbb{E}\left( \int_{0}^T \left(C_k(N_k^\phi(t), S_k^\phi({N}_k^\phi(t))) + C_k^{\infty, \phi}(N_k^\phi(t), S_k^\phi({N}_k^\phi(t))) \left. 
 - W f_k(N_k^\phi(t), S_k^\phi({N}_k^\phi(t)) )\right)dt \right. \right).
\end{equation}}
With slight abuse of notation, we  dropped the vector notations since we analyze a single bandit in the decomposed problem. The solution to (\ref{relaxation}) is obtained by combining the solution to the $K$ separate optimization problems (\ref{subproblem_k}). Under a stationarity assumption, we can invoke ergodicity to show that (\ref{subproblem_k}) is equivalent
to minimizing
\begin{equation}\label{ergodic_subproblem_k}
\mathbb{E}(C_k(N_k^\phi, S_k^\phi(N_k^\phi)))~+\mathbb{E}(C_k^{\infty, \phi}(N_k^\phi, S_k^\phi(N_k^\phi)))  - W \mathbb{E}(f_k(N_k^\phi, S_k^\phi(N_k^\phi))),
\end{equation}

where $N_k^\phi$ is distributed as the stationary distribution of the state of bandit $k$ under policy $\phi$. If $f_k(N_k^\phi, S_k^\phi({N}_k^\phi)) =1- S_k^\phi({N_k^\phi})$, the Lagrange multiplier $W$ can be interpreted as subsidy for passivity. 
 For  ease of notation, we denote the total expected cost of bandit $k$ under policy $\phi$ by
$$\mathbb{E}(T_k(N_k^\phi, S_k^\phi(N_k^\phi))) := \mathbb{E}(C_k(N_k^\phi, S_k^\phi(N_k^\phi)))~+\mathbb{E}(C_k^{\infty, \phi}(N_k^\phi, S_k^\phi(N_k^\phi))). $$

\subsection{Indexability and Whittle's Index}

Indexability is a property that allows us to develop a heuristic for the original problem. This property imposes that as the Lagrange multiplier, $W$, increases, the collection of states in which the optimal action is passive increases. 

\begin{defn}
A bandit is indexable if the set of states in which passive is an optimal action in~(\ref{ergodic_subproblem_k}) (denoted by $D_k(W)$) increases in $W$, that is, $W' <W \Rightarrow D_k (W') \subseteq D_k(W )$.
\end{defn}
Note that in case the set of states $D_k(W)$ is decreasing in $W$, one can simply switch the role of active and passive for bandit~$k$, and   the bandit would be  indexable. If the RMABP is  indexable, Whittle's index in state $N_k$ is defined as follows:
\begin{defn}
When a bandit is indexable, Whittle's index in state $n_k$ is defined as the smallest value for the subsidy such that an optimal policy for (\ref{ergodic_subproblem_k}) is indifferent of the action in state $n_k$.  Whittle's index is denoted by $W_k(n_k)$.
\end{defn}

Given that the indexability property holds, \cite{whittle1988restless} established that the solution to the
relaxed control problem (\ref{relaxation}) will be to activate all bandits that are in a state $n_k$ such that their
Whittle's index exceeds the Lagrange multiplier, i.e., $W_k (n_k) > W$. 
This optimal solution to the relaxed problem might be infeasible for the original model where the constraint (\ref{main_const}) has to be satisfied at each decision epoch. Hence, we have the following index-based heuristic for the original problem with the general resource constraint (\ref{main_const}). 
\begin{defn}[Index-based policy]\label{index_heuristic} Assume at time $t$ we are in state $\vec{N}(t) = \vec{n}$. The index-based heuristic will activate bandits in a greedy
manner: activate the bandit with the highest non-negative Whittle's index, $W_k(n_k)$,  until either the constraint is met or there are no more bandits with non-negative index to activate. 
%
%
\end{defn}

In our definition of the index-based policy, bandits with a negative Whittle's index are not activated.  This is a direct consequence of the relaxed optimization problem: when the Whittle's index is negative in state ${m}$, i.e., $W_k ({m}) < 0$,  bandit~$k$ in state $m$ is made active only if $W < W_k ({m})$. Since $W_k ({m}) < 0$, the bandit is activated only if a cost $W$ is paid for being passive.
In case $f_k(N_k^\phi, S_k^\phi(\vec{N}_k^\phi)) =S_k^\phi({N_k^\phi})$, this definition reduces to the Whittle's index policy as studied in classical RMABPs. 
In that setting, the sample-path constraint~\eqref{main_const} can also have a strict equality sign, in which case one can simply adapt the index-based policy by  activating the $M$ bandits  with the highest indices, even though these indices can be negative.  
We further note that in case the function $f_k(\cdot, \cdot)$ represents the expected capacity occupation,  the index-based rule as defined above has been numerically shown to perform well  in \cite{graczova2014generalized}. 

 \begin{rmk}
 \label{single_bandit_optimality}
In the case of a single bandit, i.e., $K=1$,  with no sample-path constraint on activation, the index-based  policy provides an \emph{optimal solution} to the original problem.
One can choose a uniformly-bounded constraint function $f(n,a)$ in such a way that the bandit is indexable. When $W=0$, the objectives of the original and the relaxed problem are the same. Hence, the index-based policy (activate the bandit whenever it is in a state~$n$ with $W(n)>0$) results in an optimal solution for the original problem for the single bandit.  
In Section~\ref{infinite_transition_rates} this is discussed for a content delivery network.
\end{rmk}

\subsection{Expression of Whittle's index}
\label{sec:al}

In this section we present an algorithm, which checks for indexability, and, if indexable, calculates  Whittle's index. This algorithm is computationally expensive. In section \ref{sec:th}, we show how this algorithm simplifies in case some structural results can be shown for the relaxed problem.  Since the action set (binary actions) is finite for each state,  we can restrict our focus to deterministic stationary policies, see \cite[Theorem 11.4.8]{puterman2014markov}.

Let $X^U\subset \mathcal{N}^d$ be the set consisting of all states bandit~$k$ can be in. For a given subset $X \subset X^U$, let policy~$X$ refer to the policy that keeps the bandit passive in all states~$x\in X$. We denote the empty set by $\varepsilon$.

Algorithm~\ref{alg} can be explained as follows. In the relaxed subproblem, the objective is to find a policy for bandit~$k$ that minimizes~\eqref{ergodic_subproblem_k}.  Note that~\eqref{ergodic_subproblem_k} is a linear function in~$W$.
  Hence, when comparing two policies~$\phi_1$ and $\phi_2$,  defined by their sets of {passive} states $X^{(\phi_1)}$ 
and $X^{(\phi_2)}$, respectively,  there is a value for the subsidy, $W_{1,2}$, for which the costs~\eqref{ergodic_subproblem_k} under both policies are equal. This is given by 
$$W_{1,2}:= \frac{
\mathbb{E}(T_k(N_k^{X^{(\phi_1)}}, S_k^{X^{(\phi_1)}}({N}_k^{X^{(\phi_1)}}))) - \mathbb{E}(T_k(N_k^{X^{(\phi_2)}}, S_k^{(\phi_2)}
({N}_k^{(\phi_2)})))
}
{\mathbb{E}(f_k(N_k^{X^{(\phi_1)}}, S_k^{X^{(\phi_1)}}({N}_k^{X^{(\phi_1)}}))) - \mathbb{E}(f_k(N_k^{(\phi_2)}, S_k^{(\phi_2)}({N}_k^{(\phi_2)})))}.
$$ 
When calculating the indices, one needs to find the policies (and their corresponding passive sets) that minimize~\eqref{ergodic_subproblem_k} for each $W$. As a first step,  one can start by searching for the optimal policy when $W=-\infty$ (the set of passive states corresponding to this optimal policy is denoted by $X_{0}$) and when $W=\infty$ (corresponding policy denoted by $X_{0'}$).  Recursively, one can find the switching point $W$ where the optimal passive set changes. That is, in Step~$j$, given it is known that the policy described by the passive set $X_{j-1}$ is optimal when $W_{j-2}<W\leq W_{j-1}$, find the crossing point  where another policy becomes optimal. Such a crossing point is denoted by $W_j$ and the corresponding optimal passive set by $X_j$. Now, if $X_{j-1}\subset X_j$, for all $j$, this implies indexability and the index for any state in $X_{j}\backslash X_{j-1}$ is given by $W_j$.

\begin{algorithm}                     
 \caption{{Indexability check and Whittle's index computations}}\label{alg}
\begin{algorithmic} [1]                   
\REQUIRE  $X^U, f_k(.), T_k(.)$

\hspace{-0.61cm}\textbf{Initialization:} Define $X_0$ as the union of sets  that equals $\arg\inf\limits_{X \subset {X^U}} \mathbb{E}(f_k(N_k^{X}, S_k^X({N}_k^{X})))$. 
Define $X_{0^{'}}$ as the union of sets that equals  $\arg\sup\limits_{X \subset {X^U}} \mathbb{E}(f_k(N_k^{X}, S_k^X({N}_k^{X})))$. 
 
\hspace{-0.61cm}\textbf{Step $j$:} Compute 
$$W_j =   \inf_{\substack{X \subset X^U }} \frac{\mathbb{E}(T_k(N_k^{X}, S_k^{X}({N}_k^{X}))) - \mathbb{E}(T_k^{X_{j-1}}(N_k^{X_{j-1}}, S_k^{X_{j-1}}({N}_k^{X_{j-1}})))}{\mathbb{E}(f_k(N_k^{X}, S_k^{X}({N}_k^{X}))) - \mathbb{E}(f_k(N_k^{X_{j-1}}, S_k^{X_{j-1}}({N}_k^{X_{j-1}})))}, j\ge 1.$$ 
\begin{description}
\item[\bf Step $j_1$:] If the infimum is not attained in any finite set, go to step $j'$.  Otherwise, let $\mathcal{X}_j$ denote the collection of sets that reaches the infimum. Define  $X_{j}$ as the union of all sets  $arg\max\limits_{Y\in \mathcal{X}_j} \mathbb{E}(f_k(N_k^{Y}, S_k^Y({N}_k^{Y})))$. 
\item[\bf Step $j_2$:] If $X_{j-1}\subset X_{j}$, set  $W_k (x) :=W_j$ for all states $x \in X_j \backslash X_{j-1}$. Else, the  system is not indexable and stop. 
\item[\bf Step $j_3$:] If $X_j = X^U$, then the system is indexable and stop. Otherwise go to step $j+1$.
\end{description}
\hspace{-0.61cm}\textbf{Step $j{^{'}}$:} Compute 
$$W_{j^{'}} = {\sup_{\substack{X \subset X^U } }
}
 \frac{\mathbb{E}(T_k(N_k^{X}, S_k^{X}({N}_k^{X}))) - \mathbb{E}(T_k^{X_{j^{'}-1}}(N_k^{X_{j^{'}-1}}, S_k^{X_{j^{'}-1}}({N}_k^{X_{j^{'}-1}})))}{\mathbb{E}(f_k(N_k^{X}, S_k^{X}({N}_k^{X}))) - \mathbb{E}(f_k(N_k^{X_{j^{'}-1}}, S_k^{X_{j^{'}-1}}({N}_k^{X_{j^{'}-1}})))}, j\ge 1.$$ 
where the supremum is taken over all those $X \subset X^U$ such that the denominator is non-zero. 
\begin{description}
\item[\bf Step $j_1'$:] If the supremum is not attained in any finite set, Whittle's index is not found. Otherwise, let $\mathcal{X}_{j'}$ denote the collection of sets that reaches the infimum. Define $X_{j^{'}}$ as the union of all sets $arg\min\limits_{Y\in \mathcal{X}_{j'}} \mathbb{E}(f_k(N_k^{Y}, S_k^Y({N}_k^{Y})))$.

\item[\bf Step $j_2'$:] If $X_{j^{'}-1}\subset X_{j^{'}}$, set $W_k (x) :=W_{j^{'}}$ for all states $x \in X_{j^{'}} \backslash X_{j^{'}-1}$, Else, the system is not indexable and stop.

\item[\bf Step $j_3'$:]  If $X_{j^{'}} = X^U$ then the system is indexable, and stop. Otherwise go to step $j'+1$.
\end{description}

%

\end{algorithmic}
\end{algorithm}

\begin{rmk}
\label{rem: alg}
In  Algorithm~\ref{alg}, an indexability check is included. 
 Note that in case indexability had been proved independently, the algorithm would simplify, as one can   replace in step~$j$ the  ``$\inf_{X\subset X_U}$'' by  ``$\inf_{X: X_{j-1}\subset X \subset X_U}$'', and similarly in step~$j'$.
\\
In addition, in Algorithm~\ref{alg}  structural properties of the optimal solution of the relaxed optimization problem are not taken into account. In Section~\ref{sec:th} we will describe how a structural property {(threshold optimality)} can help in the calculation of the index. \end{rmk}

\begin{rmk} {Since we consider the average cost criterion, the calculations in step $j$ and $j'$ in Algorithm \ref{alg} are feasible provided the steady-state distributions are known. If instead we had considered the total discounted cost criterion, the latter would depend on the initial state and transient behavior, and hence the calculation as done in Algorithm \ref{alg} is not possible. 
}
\end{rmk}

\section{Threshold policies}
\label{sec:th}
For certain one-dimensional problems, it can be established that the structure of the optimal solution of problem (\ref{ergodic_subproblem_k}) is of threshold type. That is, there is a threshold function $n_k(W)$ such that when bandit $k$ is in a state $m_k \leq n_k (W)$, then action $a$ is optimal, and otherwise action $a'$ is
optimal, $a,~ a' \in \{0, 1\}$ and $a \neq a'$. We let policy $\phi = n$ denote a threshold policy with threshold $n$, and we
refer to it as 0-1 type if $a = 0$ and $a' = 1$, and 1-0 type if $a = 1$ and $a' = 0$. 
In general it can be hard to verify whether an optimal solution is of threshold type.  
In this section we provide sufficient conditions for threshold optimality and show how Algorithm \ref{alg} simplifies.

The following result characterizes sufficient conditions on the transition rates   $q_k^a(\cdot,\cdot)$ and the jump probabilities $p_k^a(\cdot,\cdot)$ such that a  threshold policy solves problem (\ref{ergodic_subproblem_k}).  
For example, condition $(i)$ can be interpreted as follows. If there are no upward jumps under the active action or under an impulse control, and there can be an upward jump of at most one under the passive action, then 0-1 type of threshold policies are optimal. 
The proof can be found in the Appendix. 

\begin{prop}\label{threshold_optimality}
Assume $N_k(t)\in \mathcal{N}$ If one of the following conditions holds, 
\begin{enumerate}[label=(\roman*)]
\item $q_k^1(N, N+i) = 0,~\forall~i\ge 1$,  $ \ q_k^0(N, N+i) = 0,~\forall~i\ge 2$, and   $p_k^a(N, N+i ) =0,~\forall~i\ge 1, a =0, 1,$
\item $q_k^0(N, N-i) = 0,~\forall~i\ge 1$,  $ \ q_k^1(N, N-i) = 0,~\forall~i\ge 2$, and   $p_k^a(N, N-i ) =0,~\forall~i\ge 1, a =0, 1,$
\end{enumerate} 
 then there exists an $n_k \in \{-1, 0, 1, ...\}$ such that a 0-1 type of threshold policy with threshold $n_k$, optimally solves problem (\ref{ergodic_subproblem_k}). 
 
 Alternatively, if one of the following conditions holds, 

\begin{enumerate}
\item[(iii)]  $q_k^1(N, N-i) = 0,~\forall~i\ge 1$, $ \ q_k^0(N, N-i) = 0,~\forall~i \ge 2$ and  $p_k^a(N, N-i) =0,~\forall~i\ge 1, a =0, 1,$
\item[(iv)] $q_k^0(N, N+i) = 0,~\forall~i\ge 1$,  $ \ q_k^1(N, N+i) = 0,~\forall~i\ge 2$, and   $p_k^a(N, N+i) =0,~\forall~i\ge 1, a =0, 1,$
\end{enumerate}
then, there exists an $n_k \in \{-1, 0, 1, ...\}$ such that a 1-0 type of threshold policy with threshold $n_k$, optimally solves problem (\ref{ergodic_subproblem_k}).  
\end{prop}

Optimality of a threshold policy  for the relaxed optimization problem has been proved for several RMABPs, several examples can be found in \cite[Section 6.5]{gittins2011multi}. All applications as presented in this paper, fit the sufficient conditions, and hence have the threshold structure. We further note that there are models where the optimality of threshold policies has been established, but whose model parameters do not fall within the conditions of Proposition~\ref{threshold_optimality},  see for example \cite{ansell2003whittle, glazebrook2009index,  larranaga2015efficient}.

When an optimal solution for problem (\ref{ergodic_subproblem_k}) is of threshold type, we have the following sufficient condition for indexability.    For the proof we refer to the Appendix.

\begin{prop}\label{rmk_on_indexability}
If an optimal solution of (\ref{ergodic_subproblem_k}) is of threshold type, and $\mathbb{E}(f_k(N_k^n, S_k^n({N}_k^n)))$ is non-negative and strictly increasing   in $n$, then problem (\ref{ergodic_subproblem_k}) is indexable. 
\end{prop}

Algorithm \ref{alg} simplifies in case it is known beforehand that \emph{(i)} threshold policies are optimal for the relaxed problem, and \emph{(ii)} the bandit is indexable. For example, assume it is known that 0-1 type of threshold policies are optimal. From Remark~\ref{rem: alg}, one can restrict in Step~$j$ and Step~$j'$ the search to sets $X$ of the form $\{m\leq n\}$  where $n>n_{j-1}$, and initialization step reduces to
finding $n_{0} = \arg\inf\limits_{n \in \mathbb{N}} \mathbb{E}(f_k(N_k^{n}, S_k^n({N}_k^{n})))$ and $X_0=\{m\leq n_0\}$.  Similar for $X_{0'}$.

Moreover, in case the monotonic nature of the function  
\begin{equation}\label{general_index_value}
\frac{\mathbb{E}(T_k^n(N_k^n, S_k^n({N}_k^n))) - \mathbb{E}(T_k^{n-1}(N_k^{{n-1}}, S_k^{{n-1}}({N}_k^{{n-1}})))}{\mathbb{E}(f_k(N_k^{n}, S_k^n({N}_k^{n}))) - \mathbb{E}(f_k(N_k^{{n-1}}, S_k^{{n-1}}({N}_k^{{n-1}})))}
\end{equation}
can be proven, the computation of Whittle's index further  simplifies:

\begin{prop}\label{proposition:monotone_index}
Assume an optimal solution of (\ref{ergodic_subproblem_k}) is of threshold type and bandit~$k$ is indexable.  If~\eqref{general_index_value} is a monotone function in $n$, then  Whittle's index $W_k (n)$ is given by (\ref{general_index_value}).

{
Further, it follows that Whittle's index is non-decreasing if 0-1 type of threshold policies are optimal and is non-increasing if 1-0 type of threshold policies are optimal. }
\end{prop}

To illustrate our unifying framework, in the next two sections, we briefly present a few specific applications when the evolution of the bandit is driven by a Markov chain with finite and infinite transition rates. {Full details are given in the Appendices.}





\section{Applications: finite transition rates}\label{finite_transitions_examples}
\label{sec:finite}
There exists a large set of papers that calculate the Whittle's index for RMABPs with finite transition rates and time-average performance objective, see Section \ref{sec:relatedwork} for further references. 
In Section~\ref{sec:machine_repairman}, we consider one such application (machine repairman problem) to illustrate the applicability of our results. 
We however emphasize that all indices as derived in previous literature, can be derived directly from our unifying framework.

\subsection{Machine repairman problem}\label{sec:machine_repairman}
We consider the classical machine repairman problem with $\mathcal{M}$ non-identical machines and $\mathcal{R}$ repairmen, where $\mathcal{R}\le \mathcal{M} $. Any number of repairmen may be active at any decision epoch, but no more than one may work on any individual machine at a time. 

\cite{glazebrook2005index} modeled the machine repairman problem as an RMABP in \emph{discrete time} and obtained Whittle's index for the average-cost criterion by first considering the discounted-cost criterion and then letting the discounting factor tend to one.  
In this section, we describe the machine repairman model as an RMABP in \emph{continuous time}, and allow state dependent transition rates with a general cost function. We use  Proposition~\ref{proposition:monotone_index} to derive Whittle's index for the average-cost criterion, and as special cases retrieve the indices of \cite{glazebrook2005index}. 
This generalizes the model in \cite{glazebrook2005index}, where the analysis was restricted to constant cost structure (for model 2) and state-independent transition rates. 



The state of the system at epoch $t$ is an $\mathcal{M}$-dimensional vector $X(t) = \{X_1(t), X_2(t), ..., X_{\mathcal{M}}(t)\},$ where $X_k(t)$ is the state of machine $k\leq \mathcal{M}$, and $X_k(t)\in \{0,1,\ldots\}$. The machine state measures the degree of deterioration and is assumed to evolve independently of the states of the other machines. When one of the $\mathcal{R}$ repairmen is working on machine~$k$, it makes a transition from state~$n_k$ to the pristine state~0 at repair rate $r_k(n_k)$.
If instead machine~$k$ is unattended, its state deteriorates from state $n_k$ to state $n_k+1$ after an exponential amount of time with deterioration rate  $\lambda_k(n_k)$.
In addition, if machine~$k$ in state~$n_k$ is unattended, it experiences a catastrophic breakdown at rate $\psi_k(n_k)$, after which the machine is replaced   by a new machine at a considerable lump cost $L_k^b(n_k)$. 

Let $L_k^r(n_k)$ be the lump cost of using the repairman in state $n_k$, where typically $L_k^b(n_k) >> L_k^r(n_k)$. 
Let $C_k^{d}(n_k)$ denote the per unit cost of deterioration for an unattended machine $k$  in state $n_k$. 
The objective is to deploy the repairmen in order to  minimize the total long-run average cost.

This problem can be cast in the RMABP framework as follows. 
Each machine is a bandit.  At each decision epoch, at most  $\mathcal{R}$ bandits/machines can be activated (repairman is send to it). A bandit is active ($a=1$) if a repairman is deployed to this  machine, and passive ($a=0$) otherwise.    
The machine repairman problem  is hence characterized by the following transition rates: 
\begin{equation}
\label{eq:trmr}
q_k^1(n_k,0) = r_k(n_k), \ q_k^0(n_k,0) = \psi_k(n_k) \text{ and }q_k^0(n_k,n_k+1) = \lambda_k(n_k),
\end{equation}
and the cost functions under action $a=0$ and $a=1$ are given by, 
\begin{eqnarray*}
&& C_k(n_k,0)=  \psi_k(n_k)L_k^b(n_k)+C_k^d(n_k),\\
&& C_k(n_k,1)=  r_k(n_k)L_k^r(n_k).
\end{eqnarray*}
%
We set $f_k(N_k^\phi, S_k^\phi(\vec{N})) = {S_k^\phi(\vec{N})}$, so that  
the constraint $\sum\limits_{k=1}^
\mathcal{M} f_k(N_k^\phi, S_k^\phi(\vec{N})) = \sum\limits_{k=1}^
\mathcal{M} S_k^\phi(\vec{N}) \le \mathcal{R},$
ensures that at most $\mathcal{R}$ repairmen are active at a time.

The transition rates~\eqref{eq:trmr} satisfy the sufficient conditions as presented in Proposition~\ref{threshold_optimality}, hence an optimal policy of the relaxed optimization problem (\ref{subproblem_k}) is of threshold type with 0-1 structure. Further, $\mathbb{E}(f_k(N_k^{n_k}, S_k^{n_k}({N}_k^{n_k}))) = \sum\limits_{m=0}^{n_k}\pi_k^{n_k}(m)$, with $\pi_k^{n_k}(m)$ as the stationary probability of being in state $m$ under threshold $n_k$. In addition, it can be verified that $\sum\limits_{m=0}^{n_k}\pi_k^{n_k}(m)$ is strictly increasing in $n_k$ if $r_k(n_k) \le r_k(n_k+1)$ for all $n_k$ (see Appendix~\ref{MRP_indexability}). Thus, the indexability follows from Proposition \ref{rmk_on_indexability}. 
Together with Proposition~\ref{proposition:monotone_index}, we obtain the following   closed-form expression for Whittle's index. We refer to Appendix~\ref{machine_index_general_model} for the proof. 

%
%
%

\begin{prop}\label{machine_index_general_model}
Assume  $r_k(n) \le r_k(n+1),~\forall~n$ {and $r_k(1)>0$}. It holds that bandit~$k$ is indexable.  
Consider
\begin{equation}\label{general_index}
\frac{ \left(C_{Sum}(n) + L_k^r(n+1)P_{n}\right)\left(P_{Sum}(n-1) + \frac{P_{n-1}}{r_k(n)}\right)- \left(C_{Sum}(n-1)  + L_k^r(n)P_{n-1}\right) \left(P_{Sum}(n) + \frac{P_{n}}{r_k(n+1)}\right)}{\frac{P_{n-1}}{r_k(n)}\sum\limits_{i=0}^{n}\frac{P_i}{\lambda_k(i)} - \frac{P_n}{r_k(n+1)}\sum\limits_{i=0}^{n-1}\frac{P_i}{\lambda_k(i)}},
\end{equation}
where $P_0 := 1, P_i := \prod\limits_{j=1}^ip_k(j)$, $p_k(j) := \frac{\lambda_k(j)}{\lambda_k(j) + \psi_k(j)}$, $P_{Sum}(n) :=\sum\limits_{i=0}^{n}\frac{P_i}{\lambda_k(i)}$ and $C_{Sum}(n) := \sum\limits_{i=1}^{n} \left[(P_{i-1} - P_i)L_k^b(i) + \frac{P_iC_k^{d}(i)}{\lambda_k(i)} \right]$.
If~\eqref{general_index} is monotone in $n$, then Whittle's index is given by~\eqref{general_index}.
\end{prop}



For the parameter settings such that \eqref{general_index} is monotone, the Whittle's index is given by \eqref{general_index}, for the others, the index can be computed by Algorithm \ref{alg}.

For specific choices of the parameters, Whittle's  indices were previously obtained in~\cite{glazebrook2005index}. We now show how they  follow directly from~\eqref{general_index}. 
In the first model considered in~\cite{glazebrook2005index}, there are no catastrophic breakdowns,  only deterioration costs are considered,  repair rates  and repair cost are state-independent. That is, $r_k(n)=r_k$, 
$\psi_k(n) = 0$,  $L_k^b(n) = 0$, $L_k^{r}(n) =  L_k^{r}$, $\forall ~n$.
In that case,~\eqref{general_index} reduces to 
\begin{equation}
\label{eq:why}
W_k(n)= r_k\left[ \sum\limits_{i=0}^{n-1}\frac{C_k^{d}(n) - C_k^{d}(i)}{\lambda_k(i)} + \frac{C_k^{d}(n) - r_k L_k^{r}}{r_k}\right],
\end{equation}
which is indeed monotone if $C_k^{d}(n)$ is an increasing sequence in $n$
(see Appendix~\ref{deterioration} for details). 

%

Recall that we consider a continuous-time model, while  \cite{glazebrook2005index} consider discrete time set-up. In particular, in the discrete-time model, a machine is repaired in one time slot. Setting $r_k =1$ (so that in our model the expected time to repair a machine equals~1) and interpreting $1/\lambda_k(n)$  as the expected time a machine (not under repair) spends  in state~$n$, Equation~\eqref{eq:why} matches the index as derived in~\cite[Corollary 1]{glazebrook2005index}. 

In the second model considered in~\cite{glazebrook2005index}, there are no deterioration cost, but  there is a lump cost for catastrophic breakdowns.
All lump costs and repair rates are state independent.
That is, $r_k(n)= r_k$, $C_k^{d}(n) = 0$, $L_k^r(n)= R_k, L_k^b(n) = B_k,~\forall~n$. 
Now,~\eqref{general_index} reduces to  
\begin{equation}
W_k(n) = \frac{B_k\left(\frac{1-p_k(n)}{r_k} - \frac{p_k(n)}{\lambda_k(n)} + \sum\limits_{i=0}^n\frac{P_i}{\lambda_k(i)} - p_k(n)\sum\limits_{i=0}^{n-1}\frac{P_i}{\lambda_k(i)}\right)}{\frac{1}{r_k}\left( \sum\limits_{i=0}^n\frac{P_i}{\lambda_k(i)} - p_k(n)\sum\limits_{i=0}^{n-1}\frac{P_i}{\lambda_k(i)} \right)} -R_k,
\end{equation}
which is indeed monotone if $\psi_k(n)$ is an increasing sequence in~$n$  (see Appendix \ref{lump_cost} for details). 

In the discrete-time model of \cite{glazebrook2005index}, it is assumed that the time it takes to change state is equal to~1. Setting $\frac{1}{r_k} =1$ and $\frac{1}{\lambda_k(n) +\psi_k(n)} = 1~\forall~n,$ (so that in the continuous-time model the expected time to change state equals 1) we retrieve 
\begin{equation}
W_k(n) = \frac{B_k\left(\hat{T}(n) - p(n)\hat{T}(n-1)-p(n)\right)}{\left(\hat{T}(n) - p(n)\hat{T}(n-1)\right)} - R_k,
\end{equation}
where $\hat{T}(n) :=  \sum\limits_{i=0}^{n}\frac{P_i}{\lambda_k(i)}$    can be interpreted as the expected duration to either reach state~$n$ or that a catastrophic breakdown occurs, when starting in state 0.  As such, we retrieve the  Whittle's index as obtained in \cite[Corollary 2]{glazebrook2005index}   for the discrete-time case. 

\section{Applications: infinite transition rates}\label{infinite_transition_rates}
\label{sec:infinite}

To the best of our knowledge, the calculation for Whittle's index has not been earlier explored for impulse control, that is, infinite transition rates. In this section, we illustrate the applicability of impulse control in the domain of congestion control in TCP and content delivery network. 

\subsection{Congestion control of TCP flows}\label{TCP}


We consider the following congestion control model for a simple version of TCP (Transmission Control Protocol), which is the algorithm that regulates the congestion on the Internet. There are $K$ flows trying to deliver packets to their destination via a bottleneck router with buffer size $B$. We assume that flow $k$ has implemented an additive-increase/multiplicative-decrease (AIMD) mechanism as in TCP. The congestion window of each flow is adapted according to received acknowledgements. For each positive acknowledgement (ACK), the congestion window is increased by the reciprocal of it's current value,  which approximately corresponds to an increase by one packet during a round-trip-time (RTT) without lost packets. We model this with Poisson arrivals of packets for flow $k$ with arrival rate $\lambda_k$, which correspond to the time period of one RTT. 
 For each negative acknowledgement (NACK), the congestion window is \emph{immediately} decreased using the formula, 
$\text{congestion window} = \max\{\lfloor\gamma_k\times \text{congestion window}\rfloor, 1\}$,
where $0 \le \gamma_k<1$ is the multiplicative decrease factor and the function $\lfloor.\rfloor$ denotes the floor function.
{Whenever the bottleneck router is full, a controller decides which flow will receive a NACK in order to maximize the average reward. We take reward as the generalized $\alpha$-fairness function $R_k^{(\alpha)}(n)$, which is a function of $n$, the number of  outstanding  packets of flow~$k$:
 \[ R_k^{(\alpha)}(n) =
  \begin{cases}
    \frac{(1+n)^{(1-\alpha)}-1}{1-\alpha}       & \quad \text{if } \alpha \neq 1,\\
    \log(n+1)  & \quad \text{if } \alpha = 1.
  \end{cases}
\]}
  The above reward is earned if the controller admits the packet to the router ($a=1$ or ACK), and 0 otherwise ($a=0$ or NACK). The parameterized family of the above generalized $\alpha-$fair reward aims at maximizing the router's total aggregated utility. Note that the parameter $\alpha$ permits to recover a wide variety of utilities such as max-min, maximum throughput and proportional fairness (see \cite{mo2000fair, altman2008generalized}).  


 We model the above scenario as an RMABP where each flow represents a bandit and the state of the bandit $k$, $n_k$, represents the number of outstanding packets. The parameters of bandit~$k$ are:
\begin{eqnarray*}
&&{\mathcal{I}_k(n_k,0)}= 1~\text{ and }  p_k^0(n_k, \max\{\lfloor\gamma_k.n_k \rfloor, 1\}) = 1,\\
&&{\mathcal{I}_k(n_k,1)}= 0~\text{ and }
q_k^1(n_k, n_k+1) = \lambda_k,
\end{eqnarray*}
where action  $a=1~(a=0)$ stands for sending ACK (NACK). We set  $f_k(N_k^\phi, S_k^\phi(\vec{N})) = N_k^\phi S_k^\phi(\vec{N})$, which  ensures that all flows that receive an ACK can send their outstanding packets through the bottleneck router with buffer size $B$, $\sum\limits_{k=1}^Kf_k(N_k^\phi, S_k^\phi(\vec{N})) \le B$. 
Setting ${\mathcal{I}_k(n_k,0)}= 1$  ensures that the congestion window of flow $k$ instantaneously changes when it receives a NACK. We further set  $C_k(n,0)=0$ and $C_k(n,1)=-R^{(\alpha)}_k(n)$.

It follows directly from Proposition~\ref{threshold_optimality} that an optimal policy of the relaxed optimization problem (\ref{ergodic_subproblem_k}) is of threshold type with 1-0 structure.  It can be verified that $\mathbb{E}(f_k(N_k^{n_k}, S_k^{n_k}({N}_k^{n_k}))) =\sum\limits_{m=0}^{n_k}m\pi_k^{n_k}(m)$ is strictly increasing in $n_k$ (see Appendix \ref{indexability_TCP} for details), so that indexability follows from Proposition \ref{rmk_on_indexability}. Together with Proposition \ref{proposition:monotone_index}, the following result characterizes a closed form expression for Whittle's index (see Appendix \ref{expression_index} for details).

\begin{lem}\label{Index_for_TCP}
Consider  
 \begin{equation}
 \label{eq:o}
  \begin{cases}
\frac{2\lambda_k(n-S)(1-(1+n)^{1-\alpha})- \sum\limits_{m=S}^{n-1}(1-(1+m)^{1-\alpha})}{(n-S)(n-S+1)(1-\alpha)}           & \quad \text{if } \alpha \neq 1,\\
    \frac{2\lambda_k\left(\sum\limits_{m=S}^{n-1}\log(1+m)- (n-S)\log(1+n)\right)}{(n-S)(n-S+1)}  & \quad \text{if } \alpha = 1,
  \end{cases}
\end{equation}
with $S = \max\{\lfloor\gamma_k.(n+1) \rfloor, 1\}$.
If~\eqref{eq:o} is monotone in $n$, then Whittle's index is given by~\eqref{eq:o}.
\end{lem} 

If monotonicity of Whittle's index cannot be established, the index can be computed using Algorithm~\ref{alg}. 

The index-based policy (Definition \ref{index_heuristic}) would be as follows:  activate/ACK all flows with highest Whittle's indices until the buffer is filled. All other flows receive a NACK. Since a NACK implies a multiplicative decrease, it is direct that  if the buffer is full, under the index policy only the flow with the smallest index value (\ref{eq:o})  will receive a NACK.

The paper \cite{avrachenkov2013congestion} studied as well the congestion control problem as a  multi-armed restless bandit framework.  This work however assumes discrete time, and as such, no infinite transition rates could be considered.  In addition, for their analysis, they require  a bound on  the number of outstanding packets per flow. \cite{avrachenkov2013congestion}  numerically verify for indexability and obtain Whittle's indices in closed form only in the case of at most three outstanding packets per flow. 
Our approach considerably simplifies the analysis, allows to prove indexability and provides an analytical expression for Whittle's index.



\subsection{Content delivery network}\label{CDN}

In this section, we adopt the basic model of a content delivery network, as explored in \cite{larranaga2015efficient}. Our primary focus is to  illustrate the optimality of Whittle's index policy for a single-armed bandit (Remark \ref{single_bandit_optimality})  with impulse control. 

In a content delivery network, the bulk of traffic (e.g.\ software updates, video content etc.) is delay tolerant. Hence, requests can be delayed and grouped, so as to be transmitted in a multi-cast mode through the network. The challenge is to balance the \emph{gains} of grouping requests in order to save transmission capacity against the \emph{risk} of not meeting the deadline of one or more jobs. {We assume that all the jobs instantaneously depart upon receiving the service.}

Let the state $n$ denote the number of waiting requests. 
Let jobs arrive according to  a Poisson process with a state dependent rate $\lambda(n)$.  
For each request, after a  state dependent exponential  expiration time with rate $\theta(n)/n$, the request abandons the system (deadline is passed). 
Upon activation of the server, all waiting requests are cleared instantaneously. The objective is to minimize the long-run average cost incurred by the waiting jobs, by abandonments, as well as the set-up cost paid upon activation of the server.

 Let  $C^h(n)$ be the state-dependent cost per unit of time that $n$ requests are held in the queue. Let $L^a(n)$ be the state-dependent penalty (lump cost) for a job abandoning the queue, which may depend on the action~$a$ chosen. Let  $L_s^\infty(n)$ be the set-up cost (lump cost) of clearing the batch of size $n$. 

We can model this as a single armed restless bandit, where action $a=1 ~(a=0)$ stands for serving (not serving) all the waiting requests. We drop the subscript $k$ since there is only one bandit. We have the following transitions:
\begin{eqnarray*}
&&{\mathcal{I}(n,0)}= 0~\text{ and }  q^0(n,n +1) = \lambda(n), \  q^0(n, n-1) = \theta(n),\\
&&{\mathcal{I}(n,1)}= 1~\text{ and } 
p^1(n, 0) = 1.
\end{eqnarray*}
We further set $f(N^\phi, S^\phi(\vec{N})) = -\mathbf{1}_{\{N^\phi \in \mathcal{M}^\phi\}}$, where $\mathcal{M}^\phi = \{ m\in \{0,1,2,...\}: S^\phi(m) = 0, S^\phi(m+1) = 1 \}$, in order to find an \emph{optimal} activation policy (see Remark~\ref{single_bandit_optimality}). Note that we need the function $f(.)$ to  depend on policy $\phi$, different to what was assumed in Section~\ref{sec:model}. We did so in order to find a constraint function $f$ for which it holds that $\mathbb{E}(f(N^{n}, S^{n}({N}^{n})))$ is strictly increasing. Taking for example $f(n,a)=1-a$, would result in $\mathbb{E}(f(N^{n}, S^{n}({N}^{n})))=1$, since when looking to the time-average, one is always passive. 
 It can be checked that all results of this paper (and their proofs) hold true for this given choice of function $f(N^\phi, S^\phi(\vec{N})) = -\mathbf{1}_{\{N^\phi \in \mathcal{M}^\phi\}}$.

 An optimal policy of the relaxed optimization problem (\ref{ergodic_subproblem_k}) is of threshold type with 0-1 structure. The latter follows directly from Proposition \ref{threshold_optimality}. In addition, it can be verified that, if $\lambda(n)$ is non-decreasing, then  $E(f(N^n, S^n(\vec{N})))=-\pi^n(n)$  is strictly increasing in $n$, see Appendix~ \ref{CDN_appendix}.  Thus, the indexability follows from Proposition \ref{rmk_on_indexability}.  From Proposition \ref{proposition:monotone_index}, we obtain an analytical expression for Whittle's index (see Appendix \ref{CDN_appendix} for details).

\begin{lem}\label{index_result_CDN}
 Assume $\lambda(n)$ is non-decreasing. Consider 
 \begin{equation}\label{index_CDN_general}
 \frac{\sum\limits_{i=1}^{n-1}\tilde{C}(i)(\pi^n(i) - \pi^{n-1}(i)) + \tilde{C}(n)\pi^n(n) + L_s^\infty(n+1) \lambda(n)\pi^{n}{(n)} - L_s^\infty(n)\lambda(n-1)\pi^{n-1}{(n-1)}}{\pi^{n-1}{(n-1)}-\pi^{n}{(n)}},
\end{equation}
where $\tilde{C}(n) = nC^h(n) + \theta(n)L^a(n)$.
If (\ref{index_CDN_general}) is non-decreasing in $n$, then Whittle's index is given by~(\ref{index_CDN_general}). 
\end{lem}
We conclude that  an optimal policy in the content delivery problem is to activate the bandit/queue whenever the numerator of \eqref{index_CDN_general}   is   non-negative.  This follows directly  from Remark \ref{single_bandit_optimality}. 

If the rates and costs are state-independent, i.e., $\lambda(n) = \lambda,~\theta(n) = i\theta,~C^h(n)=C^h,~L^a(n)=L^a\text{ and }L_s^\infty(n)=L_s^\infty,~\forall~n$,   the optimal policy  simplifies to activating whenever $\tilde{C} \left(\mathbb{E}(N^n)- \mathbb{E}(N^{n-1})\right) - \lambda L_s^\infty (\pi^{n-1}{(n-1)}-\pi^{n}{(n)})\geq 0$, with $\tilde{C} = C^h + \theta L^a$. This coincides with the results obtained in~\cite[Proposition 3]{larranaga2015efficient}. 


\section{Conclusions and future research}
In the main contribution of this work, we derive {an analytical} expression for Whittle's index under the   average-cost criterion when each bandit evolves as a continuous-time Markov chain with possible impulse control.
 The Whittle's index is given in a compact expression, which makes the implementation of the index heuristic easier. We show that with this general formula, we can   retrieve    Whittle's index for many application, that were previously derived in the literature on a case-by-case basis.


\section*{Acknowledgement}{
This research is partially supported by the French Agence Nationale de la Recherche (ANR) through the project ANR-15-CE25-0004 (ANR JCJC
RACON) and  by ANR-11-LABX-0040-CIMI within the program 
ANR-11-IDEX-0002-02.}

\bibliographystyle{plainnat}
\bibliography{references}

\newpage
\appendices


\section{Proof of Propositions}
In this section we provide the proofs of different propositions.  
For ease of notation, we removed the subscript $k$ from all the proofs in this section. 

\subsection{Proof of Proposition \ref{threshold_optimality}}

\begin{proof}{}
Since $\mathcal{U}_{REL}$ is non-empty,   there exists a stationary optimal policy $\phi^*$ that optimally solves the subproblem (\ref{ergodic_subproblem_k}) for a bandit. Define $n^*= \min\{ m\in\{ 0,1,...\}: S^{\phi^*}(m)=1 \}$. This implies $S^{\phi^*}(m)=0~\forall~m<n^*$ and $S^{\phi^*}(n^*)=1$. From the structure on the transition rates and jump probabilities in $(i)$ of Proposition~\ref{threshold_optimality}, we have 
$q_k^1(N, N+i) = 0,~\forall~i\ge 1$,  $ \ q_k^0(N, N+i) = 0,~\forall~i\ge 2$, and   $p_k^a(N, N+i) =0,~\forall~i\ge 1, a =0, 1$. The above transition structure ensures that all the states $m> n^*$ are transient. Hence $\pi^{\phi^*}(m)=0 ~\forall~ m >n^*$. Thus, the following holds under the optimal policy $\phi^*$: 
\begin{eqnarray}\nonumber
\mathbb{E}(C(N^{\phi^*}, S^\phi(\vec{N}^{\phi^*}))) &=& \sum_{m=0}^{n^*-1}C(m,0)\pi^{\phi^*}(m) + C(n^*,1)\pi^{\phi^*}(n^*),\\\nonumber
\mathbb{E}(f(N^{\phi^*}, S^\phi(\vec{N}^{\phi^*}))) &=& \sum_{m=0}^{n^*-1}f(m,0)\pi^{\phi^*}(m) + f(n^*,1)\pi^{\phi^*}(n^*),
\end{eqnarray}
and lump cost under the optimal policy $\phi^*$ is given by:
\begin{eqnarray*}
\mathbb{E}(C^{\infty, \phi^*}({N}^{\phi^*}, S^{\phi^*}(\vec{N}^{\phi^*})) = &&\sum_{\tilde{n}}\sum_{m}\mathbb{E}\left( q^{S^{\phi^*}(\vec{N}^{\phi^*})}(N^{\phi^*}, \tilde{n}) \times \mathcal{I}(\tilde{n}, S^{\phi^*}(\vec{M}^{\phi^*}(\vec{N}^{\phi^*}, \tilde{n})))\right.
\\
&&\left.\times{ p^{S^{\phi^*}(\vec{M}^{\phi^*}(\vec{N}^{\phi^*}, \tilde{n})}(\tilde{n},m) \times  L^\infty(\tilde{n}, m, S^{\phi^*}(\vec{M}^{\phi^*}(\vec{N}^{\phi^*}, \tilde{n})) }\right),
\end{eqnarray*}
%
From  Markov chain theory, the average number of times state $y$ is visited in the next decision epoch under action $a$ given the current state $x$ can be written as:
$$\lim_{N \rightarrow \infty}\frac{1}{N} \sum_{n=1}^N 1^a_{\{X_n=x,~ X_{n+1} = y\}}=\pi(x)q^a(x,y).$$
Given that the parameters satisfy (i) of Proposition~\ref{threshold_optimality}, the lump cost can equivalently be written as;
$$
\mathbb{E}(C^{\infty, \phi^*}({N}^{\phi^*}, S^{\phi^*}(\vec{N}^{\phi^*})) = \sum_{n=0}^{n^*-1}\sum_{m=0}^n\left( p^0(n,m) L^\infty(n, m, 0)\left[ \sum_{k=0, k \neq n}^{n^* -1}q^0(k,n)\pi^{\phi^*}(k) + q^1(n^*,n)\pi^{\phi^*}(n^*)\right]\right) $$
$$+ \sum_{m=0}^{n^*}\left( p^1(n^*,m) L^\infty(n^*, m, 1)\left[ \sum_{k=0}^{n^* -1}q^0(k,n^*)\pi^{\phi^*}(k)\right]\right).$$
In the above expected lump cost, the first term is the contribution in cost due to transition from states $0,1,2, \cdots, n^*-1$ and the second term is that for the transition from state $n^*$. Additionally, it exploits the fact that $\pi^{\phi^*}(m)=0 ~\forall~ m >n^*$. It follows from the expressions of the above expected costs that the long run average cost under the optimal policy $\phi^*$,
$$\mathbb{E}(C(N^{\phi^*}, S^{\phi^*}(N^{\phi^*})))~+\mathbb{E}(C^{\infty, \phi^*}(N^{\phi^*}, S^{\phi^*}(N^{\phi^*}))) - W \mathbb{E}(f(N^{\phi^*}, S^{\phi^*}(N^{\phi^*}))),$$
is the same as the long run average cost under a 0-1 type threshold policy with threshold $n^*$,
$$\mathbb{E}(C(N^{n^*}, S^{n^*}(N^{n^*})))~+\mathbb{E}(C^{\infty, n^*}(N^{n^*}, S^{n^*}(N^{n^*}))) - W \mathbb{E}(f(N^{n^*}, S^{n^*}(N^{n^*}))).$$
Thus, a 0-1 type of threshold policy with threshold $n^*$ is optimal when $(i)$ is satisfied.  The alternate rates $(ii)$ can be proven to result in 0-1 type threshold optimality along the similar lines by considering the set $\max\{ m\in\{ 0,1,...\}: S^{\phi^*}(m)=0 \}$.
%
\end{proof}

\subsection{Proof of Proposition \ref{rmk_on_indexability} }
\begin{proof}{}


We will focus on 0-1 type of threshold policies throughout the proof. The case of threshold policies of type 1-0 can be proven similarly. 
Since an optimal solution of problem (\ref{ergodic_subproblem_k}) is of threshold type for a given subsidy $W$, the optimal average cost will be $g(W) := \min\limits_n{g^{(n)}(W)}$ where 
$$g^{(n)}(W) = \mathbb{E}(T(N^n, S^n(N^n))) - W \mathbb{E}(f(N^n, S^n(N^n))).$$

We denote the minimizer of $g(W)$ by $n(W)$. Note that the function $g(W)$ is a lower envelope of affine non-increasing functions of $W$ due to the non-negative nature of $\mathbb{E}(f(\cdot ))$. It thus follows that $g(W)$ is a concave non-increasing function. 

It follows directly that the right derivative of $g(W)$ in $W$ is given by $-\mathbb{E}(f(N^{n(W)}, S^{n(W)}(N^{n(W)})))$. Since $g(W)$ is concave in $W$, the right derivative is non-increasing in $W$. Together with the fact $\mathbb{E}(f(N^n, S^n(N^n)))$ is strictly increasing in $n$, it hence follows that $n(W)$ is non-decreasing in $W$. Since an optimal policy is of 0-1 threshold type, the set of states where it is optimal to be passive can be written as $D(W) = \{m: m \le n(W)\}$. Since $n(W)$ is non-decreasing, by definition this implies that bandit $k$ is indexable.

\end{proof}

\subsection{Proof of Proposition \ref{proposition:monotone_index}}
{  

We will focus on 0-1 type of threshold policies throughout the proof. Let $\tilde{W}(n)$ be the value for subsidy such that the average cost under threshold policy $n$ is equal to that under threshold policy $n-1$. By using (\ref{ergodic_subproblem_k}), we have $\mathbb{E}(T(N^n, S^n(N^n))) - \tilde{W}(n) \mathbb{E}(f(N^n, S^n(N^n))) = \mathbb{E}(T(N^{n-1}, S^{n-1}(N^{n-1}))) - \tilde{W}(n) \mathbb{E}(f(N^{n-1}, S^{n-1}(N^{n-1})))$. Hence, $\tilde{W}(n)$ is given by, 
$$\frac{\mathbb{E}(T^n(N^n, S^n({N}^n))) - \mathbb{E}(T^{n-1}(N^{{n-1}}, S^{{n-1}}({N}^{{n-1}})))}{\mathbb{E}(f(N^{n}, S^n({N}^{n}))) - \mathbb{E}(f(N^{{n-1}}, S^{{n-1}}({N}^{{n-1}})))},$$ 
which is {the} same as (\ref{general_index_value}). Since $\tilde{W}(n)$ is monotone, it can be verified by exploiting threshold optimality that $g(\tilde{W}(n)) = g^{(n)}(\tilde{W}(n)) = g^{(n-1)}(\tilde{W}(n))$. Similarly, $g(\tilde{W}(n-1)) = g^{(n-1)}(\tilde{W}(n-1)) = g^{(n-2)}(\tilde{W}(n-1))$. Further, monotonicity of $\tilde{W}(n)$ implies the following two possibilities:
\begin{enumerate}
\item Non-decreasing nature, i.e., $\tilde{W}(n-1)\le \tilde{W}(n)$ 
\item Non-increasing nature, i.e., $\tilde{W}(n-1)\ge  \tilde{W}(n)$ 
\end{enumerate}
But  $\tilde{W}(n-1)\ge  \tilde{W}(n)$ results in a contradiction from indexability and 0-1 type of threshold optimality. Thus, $\tilde{W}(n)$ has to be non-decreasing, i.e., $\tilde{W}(n-1)\le \tilde{W}(n)$.

It follows from indexability and 0-1 type of threshold optimality that for all $W \le \tilde{W}(n)$, the set of states where it is optimal to be passive, $D(W)$, satisfies $D(W) \subseteq \{m:m\le n-1\}$. Again from indexability in a similar way, $D(W) \supseteq {\{m:m\le n-1\}}$  for all $ W \ge \tilde{W}(n-1)$. Thus, {for $ \tilde{W}(n-1) \le  W \le \tilde{W}(n)$, $\{m:m\le n-1\} \subseteq D(W) \subseteq  \{m:m\le n-1\}$ which implies} that threshold policy $n-1$ is optimal for all $\tilde{W}(n-1) \le W \le \tilde{W}(n)$ and hence $g(W) = g^{(n-1)}(W)$ for $\tilde{W}(n-1) \le W \le \tilde{W}(n)$. Hence, $\tilde W(n)$ is the smallest value of the subsidy such that activating the bandit in state~$n$ becomes optimal, that is, Whittle's index is given by $W(n) = \tilde{W}(n)$. 

}

\section{Machine repairman problem}

In this section, we provide the details to obtain the  stationary distribution, prove indexability and derive Whittle's index for two specific models of the  machine repairman problem of Section \ref{sec:machine_repairman}.

\subsection{Stationary distribution}\label{statinary_distn_MRP} In this section, we determine the stationary distribution under a  0-1 type of threshold policy $n$. Thus,  action $a=0$ is taken in states $0,1,2,\cdots, n$ and action $a=1$ in states $n+1, n+2,...$
The transition diagram for the evolution of Markov chain is shown in Figure \ref{fig:transition_diagram_model2_machine_repairman}. 
\begin{figure}[ht]
\centering
\resizebox{.65\textwidth}{!}{\input{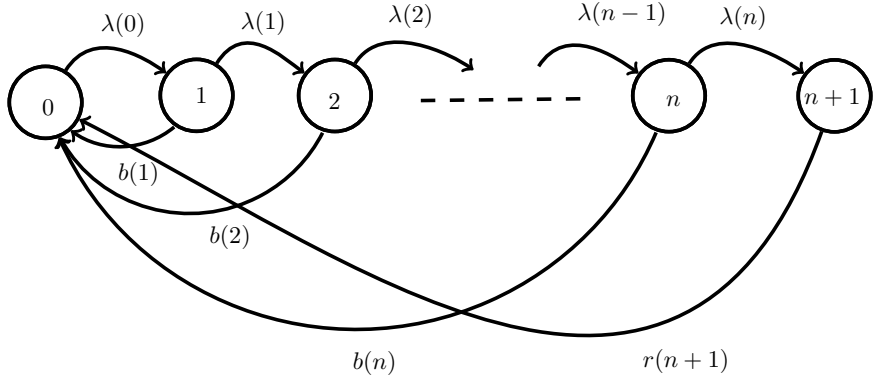}}
\caption{Transition diagram under the threshold policy $n$ for machine repairman problem}\label{fig:transition_diagram_model2_machine_repairman}
\end{figure}
The balance equations for the stationary distribution under the threshold policy $n$ are given by
\begin{eqnarray}{
\lambda(0)\pi^n(0)} &= & {\psi(1)\pi^n(1) + \psi(2)\pi^n(2) + \cdots + \psi(n)\pi^n(n)+ r(n+1)\pi^n(n+1)},  \\
\lambda(m)\pi^n(m) &=& (\lambda(m+1) + \psi(m+1))\pi^n(m+1)~\text{ for }m=0,~1,~2,...,~n-1,\nonumber \\
\lambda(n)\pi^n(n) &=& r(n+1)\pi^n({n+1}).\nonumber 
\end{eqnarray}

Using $\sum\limits_{m=1}^{n+1}\pi^n(m) = 1$, one obtains 
\begin{eqnarray}
\pi_k^{n_k}(m_k) &=& \frac{P_{m_k}}{\lambda_k(m_k)\left(\sum\limits_{i=0}^{n_k}\frac{P_i}{\lambda_k(i)} +\frac{P_{n_k}}{r_k(n_k+1)}\right)}~\forall~m_k=0,1,2,...n_k,\nonumber \\
\pi_k^{n_k}{(n_k+1)} &=& \frac{P_{n_k}}{r_k(n_k+1)\left(\sum\limits_{i=0}^{n_k}\frac{P_i}{\lambda_k(i)} +\frac{P_{n_k}}{r_k(n_k+1)}\right)}, \label{eq:f}\\
\pi_k^{n_k}(m_k) &=& 0 ~\forall~m_k = n_k+2, \cdots  \nonumber
\end{eqnarray}
where $P_i = \prod\limits_{j=1}^ip_k(j)$ and $p_k(j) = \frac{\lambda_k(j)}{\lambda_k(j) + \psi_k(j)}$; $P_0 = 1$.

\subsection{Indexability}\label{MRP_indexability}
\begin{lem}\label{machine_indexability}
Machine $k$ is indexable if the repair rates are non-decreasing in their state, i.e., $r_k(n) \le r_k(n+1)~\forall~n$, and {$r_k(1)>0$}. In particular, all machines are indexable for state-independent repair rates. 
\end{lem}

\begin{proof}{Proof.}
From Proposition \ref{proposition:monotone_index}, it follows that machine $k$ is indexable if $\mathbb{E}(f_k(N_k^{n}, S_k^{n}({N}_k^{n})))$ is strictly increasing in $n$. Recall that   $f_k(n, a) = \mathbf{1}_{\{ a = 0\}}$. Under the  0-1 type of threshold structure policy, with threshold $n$, we have 
$$\mathbb{E}(f_k(N_k^{n}, S_k^{n}({N}_k^{n}))) =\sum\limits_{m=0}^{n}\pi_k^{n}(m).$$
Thus, machine $k$ is indexable if $\sum\limits_{m=0}^{n}\pi_k^{n}(m)$ is strictly increasing in $n$. Since $\pi_k^{n}(m)=0$ for $m>n+1$, this is equivalent to proving that  $\pi_k^{n}(n+1)$ is strictly decreasing in $n$. 
From Equation~\eqref{eq:f} and some algebra, we obtain that 
$$\pi_k^{n}(n+1) - \pi_k^{n-1}(n)  = \frac{\left(\frac{\lambda_k(n)(r_k(n) -r_k(n+1))-\psi_k(n)r_k(n+1)}{\lambda_k(n) + \psi_k(n)}\right)\sum\limits_{i=0}^{n-1}\frac{P_i}{\lambda_k(i)} - r_k(n+1)\frac{ P_{n}}{\lambda_k(n)}}{r_k(n)r_k(n+1)\left(\sum\limits_{i=0}^{n}\frac{P_i}{\lambda_k(i)} +\frac{P_{n}}{r_k(n+1)}\right)\left(\sum\limits_{i=0}^{n+1}\frac{P_i}{\lambda_k(i)} +\frac{P_{n+1}}{r_k(n+2)}\right)}. $$
Note that the denominator is strictly positive. Since $r_k(n)$ is non-decreasing and {$r_k(1)>0$},  the numerator is strictly negative. That is, the result follows. 
\end{proof}

\subsection{Whittle's index: Proof of Proposition \ref{machine_index_general_model}}
Since  a 0-1 type of threshold policy is optimal, using Proposition \ref{proposition:monotone_index}, the Whittle index is given by Equation (\ref{general_index_value}), i.e., 
\begin{equation}\label{birth_death_index}
W_k(n) = \frac{\mathbb{E}(C_k(N_k^n, S_k^n(N_k^n))) - \mathbb{E}(C_k(N_k^{n-1}, S_k^{n-1}(N_k^{n-1})))}{\sum\limits_{m=0}^n\pi_k^n(m)-\sum\limits_{m=0}^{n-1}\pi_k^{n-1}(m)},
\end{equation}
if (\ref{birth_death_index}) is non-decreasing. 
The expected cost under threshold policy~$n$ in the nominator is given by 
$$\mathbb{E}(C_k(N_k^n, S_k^n(N_k^n))) = \sum\limits_{m=1}^{n}\left[\psi_k(m)L_k^b(m) + C_k^{d}(m)\right]\pi_k^n(m) +  r_k(n+1)L_k^r(n+1)\pi_k^n(n+1).$$
Using the expression for the stationary distribution as derived in Appendix \ref{statinary_distn_MRP}, we obtain that  the denominator of (\ref{birth_death_index}) simplifies to
$$\sum\limits_{i=0}^{n} \pi_k^{n}(i) - \sum\limits_{i=0}^{n-1} \pi_k^{{n}-1}(i)  = \frac{\frac{P_{n-1}}{r_k(n)}\sum\limits_{i=0}^{n}\frac{P_i}{\lambda_k(i)} - \frac{P_{n}}{r_k(n+1)}\sum\limits_{i=0}^{n-1}\frac{ P_i}{\lambda_k(i)}}{\left(\sum\limits_{i=0}^{n}\frac{P_i}{\lambda_k(i)} +\frac{P_{n}}{r_k(n+1)}\right)\left(\sum\limits_{i=0}^{n-1}\frac{P_i}{\lambda_k(i)} +\frac{P_{n-1}}{r_k(n)}\right)},$$
where $P_i = \prod\limits_{j=1}^ip_k(j)$ and $p_k(j) = \frac{\lambda_k(j)}{\lambda_k(j) + \psi_k(j)}$; $P_0 = 1$. After some algebra, we obtain that  (\ref{birth_death_index}) simplifies to the one stated in Proposition \ref{machine_index_general_model}.


\subsection{Model 1: Deterioration cost per unit}\label{deterioration}
We consider now a particular case when there are no breakdowns. Thus, $\psi_k(n_k) = 0$ and $L_k^b(n_k) = 0$. 
This simplifies since $p_k(j) = 1$ and $P_i =1$, and hence the expression in  Proposition \ref{machine_index_general_model}  simplyfies to
\begin{equation}\label{general_index_model1}
\frac{ \left(\sum\limits_{i=1}^{n} \frac{C_k^{d}(i)}{\lambda_k(i)} + L_k^{r}(n+1)\right)\left(\sum\limits_{i=0}^{n-1}\frac{1}{\lambda_k(i)} + \frac{1}{r_k(n)}\right)- \left(\sum\limits_{i=1}^{n-1} \frac{C_k^{d}(i)}{\lambda_k(i)} + {L_k^{r}(n)}\right) \left(\sum\limits_{i=0}^{n}\frac{1}{\lambda_k(i)} + \frac{1}{r_k(n+1)}\right)}{\frac{1}{r_k(n)}\sum\limits_{i=0}^{n}\frac{1}{\lambda_k(i)} - \frac{1}{r_k(n+1)}\sum\limits_{i=0}^{n-1}\frac{1}{\lambda_k(i)}}.
\end{equation}
If in addition $r_k(n)= r_k$ for all $n$, we obtain (after some algebra) 
from  Equation (\ref{general_index_model1}) that 
\begin{equation}
W_k(n)= r_k\left[ \sum\limits_{i=0}^{n-1}\frac{C_k^{d}(n) - C_k^{d}(i)}{\lambda_k(i)} + \frac{C_k^{d}(n) - r_k L_k^{r}}{r_k}\right].
\end{equation}
In addition, 
$$W_k(n) - W_k(n+1) = r_k \left[ (C_k^{d}(n) -C_k^{d}(n+1))\left(\sum\limits_{i=0}^n \frac{1}{\lambda_k(i)} + \frac{1}{r_k} \right)\right],$$
which is negative when the $C_k^d(n)$ is non-decreasing.

\subsection{Model 2: Lump cost for breakdown}\label{lump_cost}
Here, we assume that $C_k^{d}(n_k) = 0$,  $r_k(n)= r_k(n+1)=r_k$, $L_k^r(n)= R_k, L_k^b(n) = B_k~\forall~n$, and $\psi_k(n)$ is an increasing sequence. 
 From Proposition \ref{machine_index_general_model}, Whittle's index simplifies to~\eqref{general_index}.
Hence,   $W_k(n) - W_k(n+1)$ simplifies to: 
$$W_k(n) - W_k(n+1) = \frac{r_kB_k\left(\frac{P_n}{r_k} +\sum\limits_{i=0}^{n}\frac{P_i}{\lambda_k(i)} \right)\left(\frac{1}{\psi(n+1) }-\frac{1}{\psi(n)}\right)}{\left((1-p_k(n))\sum\limits_{i=0}^{n}\frac{P_i}{\lambda_k(i)} + \frac{p_k(n)P_n}{\lambda_k(n)} \right)\left((1-p_k(n+1))\sum\limits_{i=0}^{n+1}\frac{P_i}{\lambda_k(i)} + \frac{p_k(n+1)P_{n+1}}{\lambda_k(n+1)}\right)},$$
which is negative under the increasing breakdown rates assumption.


%
%
%

\section{Congestion control in TCP}

In Section \ref{TCP} we described a TCP model, where multiple users (flows) are trying to transmit packets through a bottleneck router as shown in Figure \ref{TCP_router}. 

\begin{figure}[htb!]\centering
\includegraphics[scale=0.2]{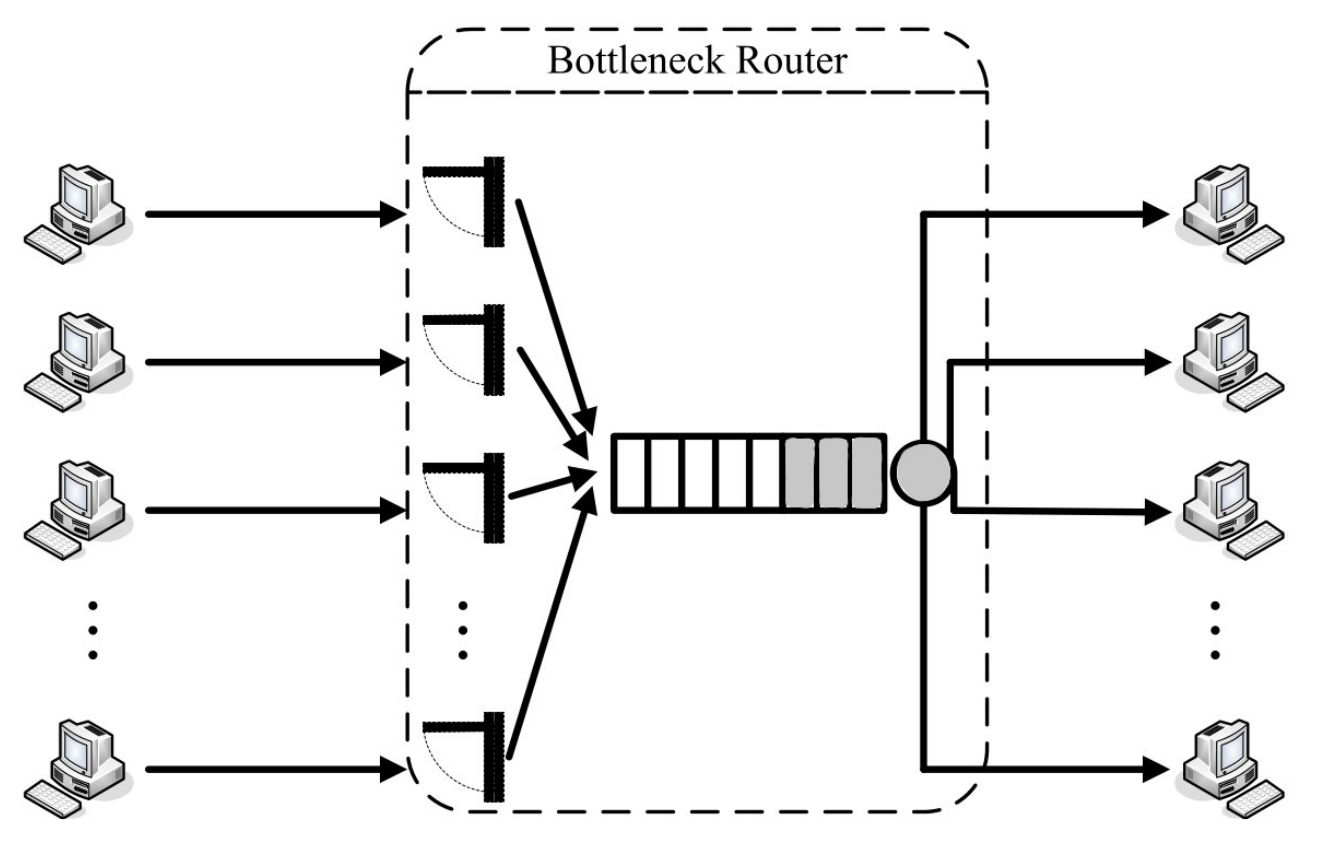}
\caption{A bottleneck router in TCP with multiple flows \citep{avrachenkov2013congestion}. }\label{TCP_router}
\end{figure}

\subsection{Stationary distribution}\label{stationary_dist_TCP}\label{indexability_TCP}

Under a 1-0 type threshold policy $n$,  action $a=1$ is taken in  states $0,1,2, \cdots, n$ and action $a=0$ in states $n+1, n+2,...$ When action $a=0$ is taken at state $n+1$, the state instantaneously changes to $S:=\max\{\lfloor\gamma.(n+1) \rfloor, 1\}$. Figure \ref{TPC_td} shows the rates and its stationary distribution is given by
\begin{eqnarray}\nonumber
\pi^n(m) &=& 0;~m = 0,1,2,...,S-1,\\\nonumber
\pi^n(m) &=& \frac{1}{n-S+1};~m = S, S+1, ...,n. 
\end{eqnarray} 

\begin{figure}[ht]
\centering
\resizebox{.65\textwidth}{!}{\input{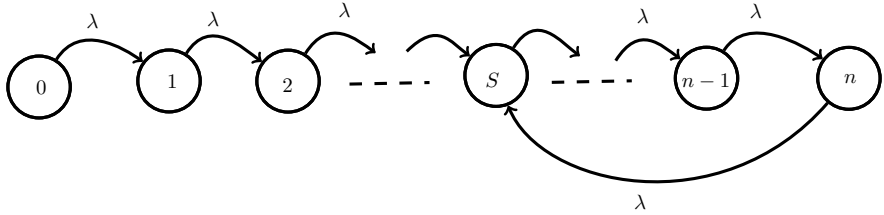}}
\caption{Transition diagram under the threshold policy `$n$' for TCP congestion control problem}\label{TPC_td}
\end{figure}

We then obtain $$\mathbb{E}(f_k(N_k^{n}, S_k^{n}({N}_k^{n})))  =\sum\limits_{m=0}^{n}m\pi_k^{n}(m)= \frac{n^2 + n - S(S-1)}{2(n-S+1)},$$
where $S = \max\{\lfloor\gamma_k.(n+1) \rfloor, 1\}$. It can be easily argued that 
\begin{equation}
\mathbb{E}(f_k(N_k^{n}, S_k^{n}({N}_k^{n}))) - \mathbb{E}(f_k(N_k^{n-1}, S_k^{n-1}({N}_k^{n-1}))) = 1/2>0.
\label{eq:ff}
\end{equation}
 Thus, $\mathbb{E}(f_k(N_k^{n}, S_k^{n}({N}_k^{n})))$ is strictly increasing in $n$ and the result follows. 

\subsection{Expression of Whittle's index: Proof of Lemma \ref{Index_for_TCP}}\label{expression_index}

Since 1-0 type of threshold policies are optimal, using Proposition \ref{proposition:monotone_index}, the Whittle index is given by
\begin{equation}\label{index_TCP}
W_k(n) = \frac{\mathbb{E}(T_k^n(N_k^n, S_k^n({N}_k^n))) - \mathbb{E}(T_k^{n-1}(N_k^{{n-1}}, S_k^{{n-1}}({N}_k^{{n-1}})))}{\mathbb{E}(f_k(N_k^{n}, S_k^n({N}_k^{n}))) - \mathbb{E}(f_k(N_k^{{n-1}}, S_k^{{n-1}}({N}_k^{{n-1}})))},
\end{equation}
if Equation (\ref{index_TCP})
is non-increasing in $n$. 

We have 
$$\mathbb{E}(T_k^n(N_k^n, S_k^n({N}_k^n))) = \sum_{m=0}^nC_k(m,1)\lambda_k \pi_k^n(m) + C_k(n+1,0)\lambda_k\pi_k^n(n+1),$$
which simplifies to
 \[ \mathbb{E}(T_k^n(N_k^n, S_k^n({N}_k^n)))  =
  \begin{cases}
\frac{\lambda_k\sum\limits_{m=S}^{n}(1-(1+m)^{1-\alpha})}{(n-S+1)(1-\alpha)}           & \quad \text{if } \alpha \neq 1,\\
    \frac{-\lambda_k\sum\limits_{m=S}^{n}\log(1+m)}{(n-S+1)}  & \quad \text{if } \alpha = 1;
  \end{cases}
\]
Together with \eqref{eq:ff} and Equation (\ref{index_TCP}), results in the Whittle index as stated in Lemma \ref{Index_for_TCP}.

\section{Content delivery network}\label{CDN_appendix}

We consider here the content delivery network as described in Section~\ref{CDN}, see also  Figure \ref{fig:single_bandit_queue}

\begin{figure}[ht]
\centering
\resizebox{.65\textwidth}{!}{
\ifx\du\undefined
  \newlength{\du}
\fi
\setlength{\du}{15\unitlength}
\begin{tikzpicture}
\pgftransformxscale{1.000000}
\pgftransformyscale{-1.000000}
\definecolor{dialinecolor}{rgb}{0.000000, 0.000000, 0.000000}
\pgfsetstrokecolor{dialinecolor}
\definecolor{dialinecolor}{rgb}{1.000000, 1.000000, 1.000000}
\pgfsetfillcolor{dialinecolor}
\pgfsetlinewidth{0.100000\du}
\pgfsetdash{}{0pt}
\pgfsetdash{}{0pt}
\pgfsetbuttcap
{
\definecolor{dialinecolor}{rgb}{0.000000, 0.000000, 0.000000}
\pgfsetfillcolor{dialinecolor}
\definecolor{dialinecolor}{rgb}{0.000000, 0.000000, 0.000000}
\pgfsetstrokecolor{dialinecolor}
\draw (27.950000\du,10.150000\du)--(33.600000\du,10.150000\du);
}
\pgfsetlinewidth{0.100000\du}
\pgfsetdash{}{0pt}
\pgfsetdash{}{0pt}
\pgfsetbuttcap
{
\definecolor{dialinecolor}{rgb}{0.000000, 0.000000, 0.000000}
\pgfsetfillcolor{dialinecolor}
\definecolor{dialinecolor}{rgb}{0.000000, 0.000000, 0.000000}
\pgfsetstrokecolor{dialinecolor}
\draw (27.900000\du,12.000000\du)--(33.615402\du,12.005402\du);
}
\pgfsetlinewidth{0.100000\du}
\pgfsetdash{}{0pt}
\pgfsetdash{}{0pt}
\pgfsetbuttcap
{
\definecolor{dialinecolor}{rgb}{0.000000, 0.000000, 0.000000}
\pgfsetfillcolor{dialinecolor}
\definecolor{dialinecolor}{rgb}{0.000000, 0.000000, 0.000000}
\pgfsetstrokecolor{dialinecolor}
\draw (33.600000\du,10.100000\du)--(33.550000\du,12.000000\du);
}
\definecolor{dialinecolor}{rgb}{1.000000, 1.000000, 1.000000}
\pgfsetfillcolor{dialinecolor}
\pgfpathellipse{\pgfpoint{40.965367\du}{10.884040\du}}{\pgfpoint{1.618107\du}{0\du}}{\pgfpoint{0\du}{1.617330\du}}
\pgfusepath{fill}
\pgfsetlinewidth{0.100000\du}
\pgfsetdash{}{0pt}
\pgfsetdash{}{0pt}
\pgfsetmiterjoin
\definecolor{dialinecolor}{rgb}{0.000000, 0.000000, 0.000000}
\pgfsetstrokecolor{dialinecolor}
\pgfpathellipse{\pgfpoint{40.965367\du}{10.884040\du}}{\pgfpoint{1.618107\du}{0\du}}{\pgfpoint{0\du}{1.617330\du}}
\pgfusepath{stroke}
\definecolor{dialinecolor}{rgb}{0.000000, 0.000000, 0.000000}
\pgfsetstrokecolor{dialinecolor}
\node at (40.965367\du,11.079040\du){Server};
\pgfsetlinewidth{0.050000\du}
\pgfsetdash{}{0pt}
\pgfsetdash{}{0pt}
\pgfsetbuttcap
{
\definecolor{dialinecolor}{rgb}{0.000000, 0.000000, 0.000000}
\pgfsetfillcolor{dialinecolor}
\pgfsetarrowsend{stealth}
\definecolor{dialinecolor}{rgb}{.000000, 0.000000, 1.000000}
\pgfsetstrokecolor{dialinecolor}
\pgfpathmoveto{\pgfpoint{23.799776\du}{8.499700\du}}
\pgfpatharc{144}{96}{5.908990\du and 5.908990\du}
\pgfusepath{stroke}
}
\pgfsetlinewidth{0.050000\du}
\pgfsetdash{}{0pt}
\pgfsetdash{}{0pt}
\pgfsetbuttcap
{
\definecolor{dialinecolor}{rgb}{.000000, 0.000000, 0.000000}
\pgfsetfillcolor{dialinecolor}
\pgfsetarrowsend{stealth}
\definecolor{dialinecolor}{rgb}{1.000000, 0.000000, .000000}
\pgfsetstrokecolor{dialinecolor}
\pgfpathmoveto{\pgfpoint{30.799776\du}{11.00\du}}
\pgfpatharc{144}{96}{6.908990\du and 8.908990\du}
\pgfusepath{stroke}
}
\pgfsetlinewidth{0.100000\du}
\pgfsetdash{}{0pt}
\pgfsetdash{}{0pt}
\pgfsetbuttcap
{
\definecolor{dialinecolor}{rgb}{0.000000, 0.000000, 0.000000}
\pgfsetfillcolor{dialinecolor}
\definecolor{dialinecolor}{rgb}{0.000000, 0.000000, 0.000000}
\pgfsetstrokecolor{dialinecolor}
\draw (31.166298\du,10.056298\du)--(31.116298\du,11.956298\du);
}
\pgfsetlinewidth{0.100000\du}
\pgfsetdash{}{0pt}
\pgfsetdash{}{0pt}
\pgfsetbuttcap
{
\definecolor{dialinecolor}{rgb}{0.000000, 0.000000, 0.000000}
\pgfsetfillcolor{dialinecolor}
\definecolor{dialinecolor}{rgb}{0.000000, 0.000000, 0.000000}
\pgfsetstrokecolor{dialinecolor}
\draw (31.931298\du,10.111298\du)--(31.881298\du,12.011298\du);
}
\pgfsetlinewidth{0.100000\du}
\pgfsetdash{}{0pt}
\pgfsetdash{}{0pt}
\pgfsetbuttcap
{
\definecolor{dialinecolor}{rgb}{0.000000, 0.000000, 0.000000}
\pgfsetfillcolor{dialinecolor}
\definecolor{dialinecolor}{rgb}{0.000000, 0.000000, 0.000000}
\pgfsetstrokecolor{dialinecolor}
\draw (32.796298\du,10.066298\du)--(32.746298\du,11.966298\du);
}
\pgfsetlinewidth{0.100000\du}
\pgfsetdash{}{0pt}
\pgfsetdash{}{0pt}
\pgfsetbuttcap
{
\definecolor{dialinecolor}{rgb}{0.000000, 0.000000, 0.000000}
\pgfsetfillcolor{dialinecolor}
\definecolor{dialinecolor}{rgb}{0.000000, 0.000000, 0.000000}
\pgfsetstrokecolor{dialinecolor}
\draw (30.411298\du,10.121298\du)--(30.361298\du,12.021298\du);
}
\pgfsetlinewidth{0.100000\du}
\pgfsetdash{}{0pt}
\pgfsetdash{}{0pt}
\pgfsetbuttcap
{
\definecolor{dialinecolor}{rgb}{0.000000, 0.000000, 0.000000}
\pgfsetfillcolor{dialinecolor}
\definecolor{dialinecolor}{rgb}{0.000000, 0.000000, 0.000000}
\pgfsetstrokecolor{dialinecolor}
\draw (29.676298\du,10.076298\du)--(29.626298\du,11.976298\du);
}
\definecolor{dialinecolor}{rgb}{0.000000, 0.000000, 0.000000}
\pgfsetstrokecolor{dialinecolor}
\node[anchor=west] at (19.100000\du,7.450000\du){State dependent arrivals, $\lambda(n)$};
\definecolor{dialinecolor}{rgb}{0.000000, 0.000000, 0.000000}
\pgfsetstrokecolor{dialinecolor}
\node[anchor=west] at (28.150000\du,15.50000\du){State dependent abandonments, $\theta(n)$};
\definecolor{dialinecolor}{rgb}{0.000000, 0.000000, 0.000000}
\pgfsetstrokecolor{dialinecolor}
\node[anchor=west] at (33.000000\du,5.500000\du){Actions: ($a=1$ or $a=0$) };
\definecolor{dialinecolor}{rgb}{0.000000, 0.000000, 0.000000}
\pgfsetstrokecolor{dialinecolor}
\node[anchor=west] at (33.000000\du,6.300000\du){$a=1:$ Activate the server};
\definecolor{dialinecolor}{rgb}{0.000000, 0.000000, 0.000000}
\pgfsetstrokecolor{dialinecolor}
\node[anchor=west] at (33.000000\du,7.100000\du){and instantaneously clear the entire queue};
\pgfsetlinewidth{0.100000\du}
\pgfsetdash{}{0pt}
\definecolor{dialinecolor}{rgb}{1.000000, 1.000000, 0.000000}
\pgfsetfillcolor{dialinecolor}
\pgfpathellipse{\pgfpoint{36.450000\du}{9.300000\du}}{\pgfpoint{0.300000\du}{0\du}}{\pgfpoint{0\du}{0.300000\du}}
\pgfusepath{fill}
\definecolor{dialinecolor}{rgb}{0.000000, 0.000000, 0.000000}
\pgfsetstrokecolor{dialinecolor}
\pgfpathellipse{\pgfpoint{36.450000\du}{9.300000\du}}{\pgfpoint{0.300000\du}{0\du}}{\pgfpoint{0\du}{0.300000\du}}
\pgfusepath{stroke}
\definecolor{dialinecolor}{rgb}{0.000000, 0.000000, 0.000000}
\pgfsetstrokecolor{dialinecolor}
\draw (35.250000\du,9.900000\du)--(37.650000\du,9.900000\du);
\definecolor{dialinecolor}{rgb}{0.000000, 0.000000, 0.000000}
\pgfsetstrokecolor{dialinecolor}
\draw (36.450000\du,9.600000\du)--(36.450000\du,11.100000\du);
\definecolor{dialinecolor}{rgb}{0.000000, 0.000000, 0.000000}
\pgfsetstrokecolor{dialinecolor}
\draw (36.450000\du,11.100000\du)--(35.250000\du,12.400000\du);
\definecolor{dialinecolor}{rgb}{0.000000, 0.000000, 0.000000}
\pgfsetstrokecolor{dialinecolor}
\draw (36.450000\du,11.100000\du)--(37.650000\du,12.400000\du);
\definecolor{dialinecolor}{rgb}{0.000000, 0.000000, 0.000000}
\pgfsetstrokecolor{dialinecolor}
\node at (36.450000\du,13.595000\du){Controller};
\end{tikzpicture}}
\caption{Optimal clearing framework as single-armed restless bandit}\label{fig:single_bandit_queue}
\end{figure}
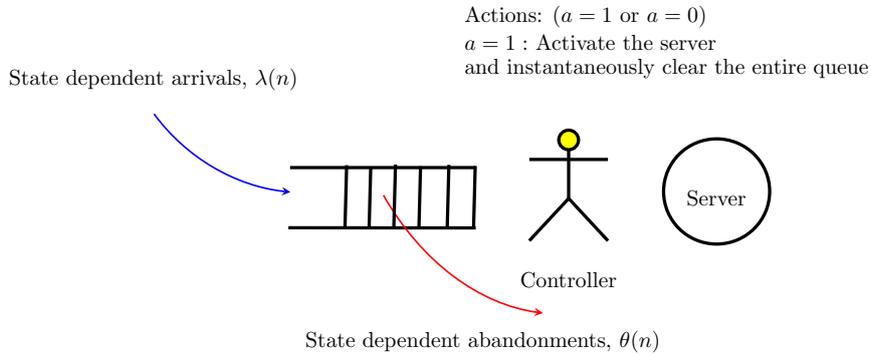

\subsection{Stationary distribution}
Under a  0-1 type of threshold policy $n$,  action $a=0$ is taken in states $0,1,2..., n$ and action $a=1$ in states $n+1, n+2,...$   The transition diagram  is shown in Figure \ref{fig:transition_diagram_CDN}.
\begin{figure}[htb!]
\centering
\resizebox{.65\textwidth}{!}{\input{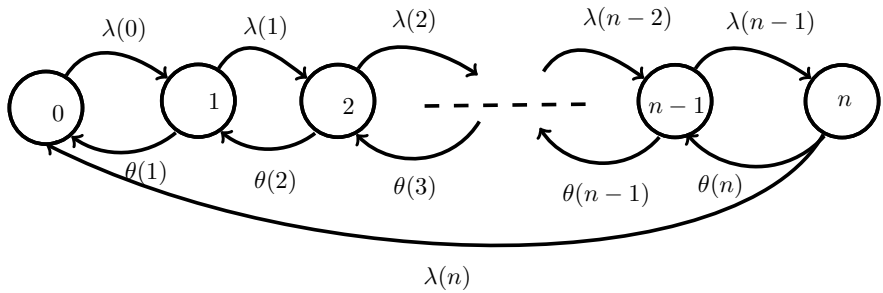}}
\caption{Transition diagram under the threshold policy $n$ in the content delivery network}\label{fig:transition_diagram_CDN}
\end{figure}

The balance equations under this chain are
\begin{eqnarray}
\pi^n(0)\lambda(0) &=& \pi^n(1)\theta(1) + \lambda(n)\pi^n(n),\\
(\lambda(k) + \theta(k))\pi^n(k) & = &\lambda(k-1)\pi^n({k-1})+ \theta(k+1)\pi^n({k+1})~\text{ for }k=1,~2,..., n-1,\\
\lambda(n-1)\pi^n({n-1}) &=& \theta(n)\pi^n(n) + \lambda(n)\pi^n(n),
\end{eqnarray}
which together with the  normalization condition $\sum\limits_{i=0}^n\pi^n(i) = 1$ results in the following stationary distribution: 
\begin{eqnarray}\nonumber
\pi^{n}(m) &=& \frac{\pi^{n}{(n)}\lambda(n)}{\lambda(m)}\left[ 1+\sum\limits_{i=1}^{n-m}p(m+1,m+i) \right]  ~\forall~m=0,1,2,...n-1,\\\label{Pi_CDN}
\pi^{n}{(n)} &=& \left(1+\sum\limits_{k=0}^{n-1}\frac{\lambda(n)}{\lambda(k)}\left[1+\sum\limits_{i=1}^{n-k}p(k+1,k+i)\right]\right)^{-1},\\
\pi^{n}(m) &=& 0 ~\forall~m = n+1, \cdots, 
\end{eqnarray}
where $p(k+1, k+i) = \dfrac{\theta(k+1)\theta(k+2)...\theta(k+i)}{\lambda(k+1)\lambda(k+2)...\lambda(k+i)}~\forall~i\ge 1$. 

{The summation term in denominator of (\ref{Pi_CDN}) is strictly increasing in $n$ if $\lambda(n)$ is non-decreasing. Thus, $\pi^{n}{(n)}$ will be strictly decreasing in $n$ under non-decreasing assumption on $\lambda(n)$.
}


 \subsection{Whittle's index: Proof of Lemma \ref{index_result_CDN}}
 

The expected cost  under threshold policy $n$ is given by
$$\mathbb{E}(T^n(N^n, S^n({N}^n))) = \sum\limits_{i=1}^{n}(iC^h(i) + \theta(i)L^a(i))\pi^n(i) +  \lambda(n)L_s^\infty(n+1) \pi^{n}(n).$$
Similarly, for the threshold policy $n-1$ 
$$\mathbb{E}(T^{n-1}(N^{n-1}, S^{n-1}({N}^{n-1}))) = \sum\limits_{i=1}^{n-1}(iC^h(i) + \theta(i)L^a(i))\pi^{n-1}(i) +  \lambda(n-1)L_s^\infty(n) \pi^{{n-1}}(n-1).$$
From Proposition \ref{proposition:monotone_index}, 
we get the expression  as stated in Lemma~\ref{index_result_CDN}. 

\end{document}